\documentclass[a4paper,11pt]{article}

\usepackage[latin1]{inputenc}
\usepackage[brazil,portuges,english]{babel}

\usepackage{graphicx,subfigure,color,epsfig,lscape,rotating}

\usepackage{varioref,float,amssymb,amsmath,amsfonts,amsthm}

\usepackage{changebar,longtable,indentfirst,booktabs,etoolbox}

\usepackage{array}

\usepackage[pdftex]{hyperref}

\newtheorem{myth}{Theorem}

\newtheorem{mylem}{Lemma}
\newtheorem{mydef}{Definition}

\def\ds{\displaystyle}

\begin{document}

\title{Analysis and numerical simulation of the nonlinear beam equation with moving ends }
\author{N.~Quintino$^{1}$,~M.~A.~Rincon$^1$\\
\thanks{E-mails:~natanaelquintino@gmail.com,~rincon@dcc.ufrj.br}\\
$^1$\small\it
    Instituto de Matem\'{a}tica, Universidade Federal do Rio de Janeiro, Rio de Janeiro, Brazil}

\date{}
\maketitle

\begin{abstract}

The numerical analysis for the  small amplitude motion of an elastic beam with internal damping is investigated in domain with moving ends. An efficient numerical method is constructed to solve this moving
boundary problem. The stability and convergence of the method is studied, and the errors of both the semi-discrete and fully-discrete schemes are derived, using Hermite's polynomials as a base function, proving that the method has order of quadratic convergence in space and time.

Numerical simulations using the finite element method associated with the finite difference method (Newmark's method) are employed, for one-dimensional and two-dimensional cases. To validate the theoretical results, tables and figures are shown comparing approximate and exact solutions. In addition, numerically the uniform decay rate for energy and the order of convergence of the approximate solution are also shown.


\end{abstract}

\section{Introduction}\label{sec:intro}

The  equation of motion of a thin beam with weak-internal damping undergoing
cylindrical bending can be written as

	\begin{equation}\label{edp:movel}
		u''(x,t) - \left(\zeta_0 + \zeta_1\int_{\Omega_t}\big|\nabla u(x,t)\big|^2 dx\right)\Delta u(x,t) + \Delta^2 u(x,t) + \nu u'(x,t) = 0, \mbox{ in } Q_t,
	\end{equation}
where $\,u\,$ is the transverse displacement, $x\in\mathbb{R}^n$ is the vector of spatial coordinates, $\,t\,$ is the time.  The aerodynamic damping term is denoted by $\,\nu,\,$  $\,\zeta_1\,$ is the nonlinear stiffness of beam, $\,\zeta_0\,$ is an in-plane tensile load, and $\,(x, t)\,$ belongs to the noncylindrical domain, defined by
\begin{equation}\label{front}
		Q_t = \big\{(x,t)\in\mathbb{R}^{n+1}; x=K(t)y,~y\in\Omega,~0<t<T\big\},
\end{equation}
 with $\Omega=(0,1)^n\subset\mathbb{R}^n$ is an open, limited and regular and
 $K:[0,\infty[\longrightarrow\mathbb{R}$ is a $C^2$- real function.
All quantities are physically nondimensionalized, $\,\nu,\,$ $\,\zeta_1\,$ are fixed positive and $\,\zeta_0$ is not necessarily positive.	

We will study the equation \eqref{edp:movel} with zero boundary conditions $\forall t\in [0, T[$,
as follows,
\begin{equation}\label{edp:condFronteiraMovel}
		u(x,t) = 0 \quad\mbox{and}\quad \nabla u \equiv 0,~\mbox{on}~\partial\Omega_t\times[0,T[,
\end{equation}
where $\partial\Omega_t = \{x=K(t)y;~y\in\partial\Omega\}$.

In the absence of the damping term $\nu~\!u_t(x,t)$, many authors have already studied Cauchy's problems and the mixed problems associated with the equation \eqref{edp:movel} on bounded and unlimited cylinders. The existence and uniqueness of solution  was demonstrated in \cite{ExistUnicBeam}.
	
	The Cauchy problem associated with the \eqref{edp:movel} equation in abstract terms on a Hilbert space is studied, among other authors, by \cite {AbstractBeam1,AbstractBeam2}. In addition, in the last
	 were established results on existence, unicity and asymptotic stability of the solution.
		
For the one-dimensional case, omitting the term $u_{xxxx}$ and considering the cylinder $Q =\Omega\times [0,\infty[$, with $\Omega\subset\mathbb {R} $ open and limited, we obtain the Kirchhoff equation with internal damping, namely,		
		
	\begin{equation}\label{intro:Kirchhoff}
		u''(x,t) - \left(\zeta_0 + \zeta_1\int_{\Omega}\big|u_x(t)\big|^2 dx\right)u_{xx}(x,t) + \nu u'(x,t) = 0, \mbox{ in } Q.
	\end{equation}
	
The equation \eqref{intro:Kirchhoff} has been extensively studied by a wide variety of authors on
$n$-dimensional cases and generic mathematical models defined in Hilbert spaces. The existence of a local and global solution can be found in several physical-mathematical contexts, such as \cite{Kirchhoff1, Kirchhoff2}.

In the article \cite {Kirchhoff2} a new model associated with the Kirchhoff equation \eqref{intro:Kirchhoff} was introduced, called the Medeiros-Kirchhoff equation, given by	
		
	\begin{align*}
		u''(x,t) - \left(a(t) + b(t)\int_{\Omega}\big|u_x(t)\big|^2 dx\right)u_{xx}(x,t) = 0, \mbox{ in } Q_t.
	\end{align*}
	
	There are several areas that work with the application of the Beam equation, for example, in civil and naval engineering, as well as in the aerospace area as the reaction of rocket and satellite metal structures to various space situations, as can be observed in the articles \cite {nasa1, nasa2}. In addition to these areas, it is possible to associate Maxwell's equations, which govern electromagnetic theory, and the Euler-Bernoulli Beam Equation, in the study presented by article \cite {maxwellBeam}

The equation \eqref{edp:movel} follows the Euler-Bernoulli Beam model. In addition to this model, there are other models, extensions of the Euler-Bernoulli model, which consider rotational movements besides transverse ones, being possible to create models closer to reality. The article \cite{studyBeamEquations} presents a more detailed study on three of these models.

The existence and uniqueness of the solution, the asymptotic decay and the numerical simulations of the beam  equation with movinf boundary were studied in \cite{Raquel} for the one-dimensional case.

In this paper, we present a family of numerical method for the one-dimensional and two-dimensional cases, based on
 the finite element method and the finite difference method. In addition, we are doing numerical analysis for both the semi-discrete problem and the fully discrete problem, showing that the order of convergence is quadratic in space and time. For this we will use the Faedo-Galerkin method, using the Hermite polynomials as base functions. In addition, the nonlinear system associated with the system of ordinary differential equations is being solved by Newton's method and for temporal discretization is applying the Neumark's method, where we will show the convergence of the approximate solution to a family of numerical methods depending on the choice of $\theta\in]1/4; 1]$.
 Numerical simulations are presented as examples, with different types of boundary and tables showing the order of numerical convergence. In addition, we present the asymptotic decay, and the graphs of the solutions for the one-dimensional and two-dimensional cases.

\section{Analytical results} \label{sec:resultAnal}

Our section will present, without demonstrating, the existence and uniqueness results of the mixed problem solution for the equation \eqref{edp:movel}, considering the boundary conditions
 \eqref{edp:condFronteiraMovel}, given by:

	\begin{equation}\label{prob:movel}
		\left\{
		\begin{aligned}
			&\begin{aligned}
				u''(x,t) - \left(\zeta_0 + \zeta_1\int_{\Omega_t}\big|\nabla u(x,t)\big|^2 dx\right)\Delta u(x,t) + \Delta^2 u(x,t) + \nu u'(x,t) = 0, \mbox{in}~ Q_t
			\end{aligned}\\	
			&~\!u(x,0) = u_0(x),~~ u'(x,0)=u_1(x),~\mbox{in}~ \Omega_0,\\
			&~\!u(x,t) = 0~\mbox{and}~ \nabla u \equiv 0,~\mbox{in}~ \partial\Omega_t\times[0,T[,
		\end{aligned}
		\right.
	\end{equation}
 where    $Q_t$ is defined in \ref{front} and  satisfying the hypotheses
	\begin{equation}\label{hip:dominio2D_H1}
		\left\{
		\begin{aligned}
			&Q_t \subset\Omega_t\times]0,T[,~ \mbox{with } \Omega_t\subset\mathbb{R}^n~ \mbox{open and limited with } 0\in{\Omega_t},\\
			&K  \in C^2\big([0,+\infty[\big),~ \mbox{with } 0<K_0\leq K(t)\leq K_1,\quad 0 < K'(t)\leq K_2,\\
		\end{aligned}
		\right.
		\tag{H1}
	\end{equation}			
	where $K_0,K_1$ and  $K_2$ are positive real constants.
	
	Let us consider the diffeomorphism $\tau : Q_t \rightarrow Q=(0,1)^n\times(0,T)$, defined
by $\tau (x, t) = (y, t)$, with $x=K(t)y$ . Therefore, the change of variable $u(x, t) = v(y, t)$  transforms the problem\eqref{prob:movel} into the equivalent problem with cylindrical boundary
	\begin{equation}\label{prob:fixa}
		\left\{
		\begin{aligned}
			&\begin{aligned}
				&v''(y,t) -b_1(t)\big|\nabla v\big|_0^2 \Delta v(y,t) + b_2(t)\Delta^2 v(y,t) + \nu~\!v'(y,t) + a^{(2)}_{ij}(y,t)\nabla_{y_i,y_j}v(y,t)\\
				&\quad- a_i^{(1)}(y,t)\Delta_{y_i} v(y,t) - a^{(3)}_i(y,t)\nabla_{y_i}v(y,t) - a^{(4)}_i(y,t)\nabla_{y_i}v'(y,t)  = 0, \mbox{ in } Q\\
			\end{aligned}\\	
			&~v(y,0) = v_0(y),~~ v'(y,0)=v_1(y)\quad\mbox{in } ~\Omega,\\
			&~v(y,t) = 0~\mbox{and}~ \nabla v \equiv 0,~\mbox{in}~ \partial\Omega\times[0,T[,
		\end{aligned}
		\right.
	\end{equation}
	where, for simplicity, we denote $K = K(t)$ and
	\begin{equation}
		\label{def:asEbs}
		\begin{aligned}
			&b_1(t) = \frac{\zeta_1}{K^4},~b_2(t) = \frac{1}{K^4},~a^{(1)}_i(y,t) = \frac{1}{K^2}\Big[\zeta _0 - 4(y_iK')^2\Big],~a^{(2)}_{ij}(y,t) = 4y_iy_j{(K'/K)^2},\\
			&a^{(3)}_i(y,t) = \frac{1}{K^2}\Big[2y_i(K')^2 - y_iK(\nu K' + K'')\Big],~a^{(4)}_i(y,t) = -2y_i({K'}/{K}).			
		\end{aligned}
	\end{equation}

In the following, we present the results on the existence and uniqueness of the weak solution of problems \eqref{prob:movel} and \eqref{prob:fixa} and asymptotic decay, whose demonstrations can be found in \cite{viga1D} for the one-dimensional case and \cite{viga2D} for the two-dimensional case.

	\begin{myth}(Existence and Uniqueness of solution in nocylindrical domain.)\label{teo:EeUmovel2D}
		
	Let $u_0\in H_0^2(\Omega_0)$, $u_1\in L^2(\Omega_0)$ and the hypothesis \eqref{hip:dominio2D_H1}. Then there is a unique weak solution $u:\widehat{Q}\longrightarrow \mathbb{R}$ for the problem
	\eqref{prob:movel} satisfying the following conditions:
		\begin{align*}
			&1.\quad u  \in  L^2(0,T;H_0^2(\Omega_t)), \\
			&2.\quad u' \in  L^2(0,T;  L^2(\Omega_t))
		\end{align*}
		and the weak formulation of \eqref{prob:movel}$_1$ is verified in the sense of
		 $L^2(0,T;H^{-2}(\Omega_t))$.		
	\end{myth} 	
	
	\begin{myth}(Existence and Uniqueness of solution in cylindrical domain)\label{teo:EeUfixa2D}
		
		Let $v_0\in H_0^2(\Omega_0)$, $v_1\in L^2(\Omega_0)$ and the hypothesis \eqref{hip:dominio2D_H1}. Then , Then there is a unique weak solution $v:Q\longrightarrow \mathbb{R}$ for the problem \eqref{prob:movel} satisfying the following conditions:
		\begin{align*}
			&1.\quad v  \in  L^2(0,T;H_0^2(\Omega)), \\
			&2.\quad v' \in  L^2(0,T;  L^2(\Omega))
		\end{align*}
		and the weak formulation of \eqref{prob:fixa}$_1$ is verified in the sense of
		 $L^2(0,T;H^{-2}(\Omega))$.		
	\end{myth}

	\begin{myth}(Asymptotic Decay: Two-dimensional case)\label{teo:decaimento2D}
		
		Consider the hypothesis \ref{hip:dominio2D_H1}, then we have the energy \eqref{prob:movel}, given by
		\begin{align*}
			E(t) = \frac{1}{2}\int_{\Omega}\Big[\big|u'(x,t)\big|^2 + \big|\Delta u(x,t)\big|^2 + \zeta_0\big|\nabla u(t)\big|_0^2+ \frac{\zeta_1}{2}\big|\nabla u(x,t)\big|^4\Big]dx,
		\end{align*}
		satisfie
		\begin{align*}
			E(t) \leq A_0~e^{-A_1t},~\quad \forall t\geq 0,
		\end{align*}
		where $A_0$ and $A_1$ are positive constants.
	\end{myth} 	
	
	From the important theoretical results, we can now develop numerical methods to determine an approximate numerical solution and verify the consistency of the numerical methods used.
	
	Due to the equivalence of the problems, the methods used will be for the problem \eqref{prob:fixa}   that is defined in a cylindrical domain.

\section{Numerical Method}\label{sec:metNum}

The purpose of this section is to develop a numerical method, based on the finite element
method and the finite difference method to make the numerical simulations in order to obtain
a  approximate solution of the problem \eqref{prob:fixa}.

Denoting the Hilbert space $H_0^2(\Omega)=V$, then weak formulation of the problem \eqref{prob:fixa} can be written in the follows form, when we integrate by parts the second, third, fifth and sixth terms:
	
		\begin{equation}\label{prob:fixaVar}
		\left\{
		\begin{aligned}
			&\big(v''(t),w\big) + b_1(t)\big|\nabla v(t)\big|_0^2 \big(\nabla v(t),\nabla w\big) + b_2(t)\big(\Delta v(t),\Delta w\big) + \nu~\!\big(v'(t),w\big) \\
			&\quad+  \big(a^{(1)}_i(t)\nabla_{y_i} v(t),\nabla_{y_i}w\big) + \big(a^{(2)}_{ij}(t)\nabla_{y_i,y_j}v(t),w\big) \\
			&\quad+ \big(a^{(4)}_i(t)\nabla_{y_i}v'(t),w\big) + \big(a^{(5)}_i(t)\nabla_{y_i}v(t),w\big)  = 0, \mbox{ in } Q\\
			&\big(v(y,0),w\big) = \big(v_0(y),w\big),~~ \big(v'(y,0),w\big) = \big(v_1(y),w\big), ~\mbox{in } ~\Omega,\\
			&\big(v,w\big) = 0~\mbox{and}~ \big(\nabla v,w\big) \equiv 0,~\mbox{in } \partial\Omega\times[0,T[,~\forall w\in V,
		\end{aligned}
		\right.
    \end{equation}
    where $\displaystyle a^{(5)}_i(t) = a^{(3)}_i(t) + 2\frac{K'}{K}a^{(4)}(t)$.		

	\subsection{Approximated Problem}	
	
Let $V_h=[\varphi_1, \varphi_1,\cdots, \varphi_m ]$  be a subspace generated by the $m$ first vectors of the basis of Hilbert space $V = H_0^2(\Omega)$. Then we want to determine an approximate solution $ v_h(t)\in V_h $ of the following approximate system, given by:
\begin{equation} \label{edp:fixaVarAprox}
                  \begin{aligned}
				&\big(v_h''(t),w_h\big) + b_1(t)\big|\nabla v_h(t)\big|_0^2 \big(\nabla v_h(t),\nabla w_h\big) + b_2(t)\big(\Delta v_h(t),\Delta w_h\big) + \nu~\!\big(v_h'(t),w_h\big) \\
				&~~+ \big(a^{(1)}_i(t)\nabla_{y_i} v_h(t),\nabla_{y_i}w_h\big) - \big(a^{(2)}_{ij}(t)\nabla_{y_i}v_h(t),\nabla_{y_j}w_h\big) 
				\\
				&~~+ \big(a^{(4)}_i(t)\nabla_{y_i}v_h'(t),w_h\big) + \big(a^{(5)}_i(t)\nabla_{y_i}v_h(t),w_h\big) = 0, ~\forall w_h\in V_h,
            \end{aligned}
        \end{equation}
        where we integrate by parts the sixth term.

	Since that, $v_h(t)\in V_h$, then the function can be represented by
	\begin{equation}\label{combLin_vh}
		v_h(y,t) = \sum_{i=1}^m d_i(t)\varphi_i(y), \quad\varphi_i(y)\in V_h.
	\end{equation}
	Taking $w_h = \varphi_l(y) \in V_h$, for some $l\in\{1,\cdots,m\}$, and substituting  \eqref{combLin_vh} in \eqref{edp:fixaVarAprox}, we obtain
	\begin{align}
		\nonumber
        &\sum_{k=1}^{m}\Big\{d''_{k}(t)\big(\varphi_{k}(y),\varphi_{l}(y)\big)
        + b_1(t)d_{k}(t)\Big|\sum_{k=1}^{m}d_{k}(t)\nabla\varphi_{k}(y)\Big|_0^2\big(\nabla\varphi_{k}(y),\nabla\varphi_{l}(y)\big)
        \\\label{metPreMat}	
        &\quad+ b_2(t)d_k(t)\big(\Delta\varphi_{k}(y),\Delta\varphi_{l}(y)\big)
        + d'_{k}(t)\Big[\nu~\!\big(\varphi_k(y),\varphi_l(y)\big)
        + \Big(a^{(4)}_i(t)\nabla_{y_i}\varphi_k(y),\varphi_l(y)\Big)\Big]
        \\\nonumber
        &\quad+ d_{k}(t)\Big[\big(a^{(1)}_i(t)\nabla_{y_i} \varphi_k,\nabla_{y_i}\varphi_l\big) - \big(a^{(2)}_{ij}(t)\nabla_{y_i}\varphi_k,\nabla_{y_j}\varphi_l\big) + \big(a^{(5)}_i(t)\nabla_{y_i}\varphi_k,\varphi_l\big)\Big]\Big\} = 0.
	\end{align}

In order to validate the method \eqref{metPreMat}, we add the term source $\big(f (y, t), w \big)$ on the right side of the equation \eqref{edp:fixaVarAprox} and, taking $w=\varphi_l $, we have the following definitions for the matrices and vectors:	
\begin{equation}\label{def:MatVet}
		\begin{aligned}
			&A_{kl} = \Big(\varphi_{k}(y),\varphi_{l}(y)\Big), \quad K_{1(kl)} = \Big(\nabla\varphi_{k}(y),\nabla\varphi_{l}(y)\Big), \quad K_{2(kl)} = \Big(\Delta \varphi_{k}(y),\Delta\varphi_{l}(y)\Big),
			\\
	     	&B_{1(ikl)}(t) = \Big(a_i^{(1)}(t)\nabla_{y_i}\varphi_{k}(y),\nabla_{y_i}\varphi_{l}(y)\Big), \quad B_{2(ijkl)}(t) = \Big(a_{ij}^{(2)}(t)\nabla_{y_i}\varphi_{k}(y),\nabla_{y_j}\varphi_l(y)\Big),\\
	     	&B_{3(ikl)}(t) = \Big(a_{i}^{(4)}(t)\nabla_{y_i}\varphi_{k}(y),\varphi_l(y)\Big), \quad B_{4(ikl)}(t) = \Big(a_{i}^{(5)}(t)\nabla_{y_i}\varphi_{k}(y),\varphi_l(y)\Big),\\
	     	&F_{l}(t) = \big(f(t),\varphi_l(y)\big), \quad d_{0(l)} = \big(v_0(y),\varphi_l(y)\big), \quad d_{1(l)} = \big(v_1(y),\varphi_l(y)\big).
		\end{aligned}
	\end{equation}
	Considering the definitions \eqref{def:MatVet} and varying $l=1,2, \cdots, m$, we obtain the following nonlinear system of second order ordinary differential equations:
		\begin{equation}\label{sistemaEDO}
		\left\{
		\begin{aligned}
	    	& A d''(t)
	    		+ G\big(t,d\big)~\!K_{1}~\!d(t)
		    	+ L_1(t)~\!d'(t)
	    	    + L_2(t)~\!d(t) = F(t),\\
	    	&d(0) = d_0 \quad\mbox{and}\quad d'(0) = d_1,
		\end{aligned}
		\right.
	\end{equation}
	with
	\begin{equation}\label{def:G_Ls}	
		\begin{aligned}
			&G\big(t,d\big) = b_1(t)\Big|\sum_{k=1}^m d_k(t)\varphi_k\Big|_0^2,\quad L_1(t) = \nu A +  \widehat{B}_{3(i)}(t),\\
			&L_2(t) = b_2(t)K_{2} +  \Big(\widehat{B}_{1(i)}(t) + \widehat{B}_{4(i)}(t) - \widehat{B}_{2(ij)}(t)\Big),
		\end{aligned}
	\end{equation}
	where we denote by  $\widehat{B}$  the transpose of the matrix $ B $.
	
Now, to solve the system of ordinary nonlinear equations \eqref{sistemaEDO} in discrete time, we will use the finite difference method.

	\subsection{The Newmark's method}\label{sistemaNewmark}
	
	In order, consider the discretize the fixed  time interval $[0,T]$,  using a uniform time step of size $ \Delta t = T/N$, then
 $ t_\eta = \eta \Delta t$, ~for $\eta=0,1,\cdots, N$ and  $\chi^\eta=\chi(\cdot,t_\eta)$. Consider the Neumark's approximations:
$$
\chi^{\eta+\theta} = \theta \chi^{\eta-1} + (1-2\theta)\chi^{\eta}+\theta \chi^{\eta+1}.
$$

Substituting  the approximations in the system \eqref{sistemaEDO} and taking the equation in the discrete time $t=t_\eta = \eta \Delta t$ , we obtain
\begin{equation}\label{sistema_n>2}
		\begin{aligned}
			&\Big(M_1^{\eta+1} + \theta~\!(\Delta t)^2 G^{\eta+1}(d)K_1\Big) d^{\eta+1}+ M_2^{\eta}d^\eta \\
			&\qquad+ \Big(M_3^{\eta-1} + \theta~\!(\Delta t)^2 G^{\eta-1}(d)K_1\Big) d^{\eta-1}
= (\Delta t)^2 F^{\eta+\theta},
		\end{aligned}
\end{equation}
or equivalently
\begin{equation}\label{sistema_n>0}
		\begin{aligned}
			&\Big(M_1^{\eta+1} + \theta~\!(\Delta t)^2 G^{\eta+1}(d)K_1\Big) d^{\eta+1}=- M_2^{\eta}d^\eta - \Big(M_3^{\eta-1} + \theta~\!(\Delta t)^2 G^{\eta-1}(d)K_1\Big) d^{\eta-1}
+ (\Delta t)^2 F^{\eta+\theta},
\end{aligned}
\end{equation}
	where
	\begin{align*}
		M_1^{\eta+1} &= A + \frac{\Delta t}{2} L_1^{\eta+1} + \theta(\Delta t)^2~\!L_2^{\eta+1},\quad M_2^{\eta} = (\Delta t)^2(1-2\theta)\big(G^\eta(d)K_1 + L_2^\eta\big) - 2A,\\
		M_3^{\eta-1} &= A - \frac{\Delta t}{2} L_1^{\eta-1} + \theta(\Delta t)^2~\!L_2^{\eta-1},\quad
		F^{\eta+\theta}=\theta F^{\eta-1}+(1-2\theta) F^{\eta}+\theta F^{\eta+1}.
	\end{align*}

\paragraph{Remark.} For the initial step, that is, for $\eta=0$, appears in the equation the term $d^{-1}$, that does not exist.
Proceeding formally, we use a quadratic order approximation for the initial data $d'(0)$. Therefore, we consider $d^{-1}= d^1 - 2\Delta t d'(0)$, so we have $d^{\eta+\theta}=2\theta d^{\eta+1} + (1-2\theta)d^\eta -2\Delta t \theta d'(0)$. So, for $\eta =0$ in \eqref{sistema_n>0}, we have
\begin{equation}\label{sistema_n=0}
		\Big[M_1^1 + M_3^0 + \theta~\!(\Delta t)^2 \Big(G^{1}(d)+G^{-1}(d)\Big)K_1\Big] d^{1}=
		- M_2^{0}d^0 + 2\Delta t~\!M_3^{0}d_1+ (\Delta t)^2 F^{0+\theta},
\end{equation}
	where we assume the approximation $M_3^{-1}=M_3^{0},\quad F^{0+\theta}= \theta F^1 + (1-\theta)F^{0}~$ and
	$$
	G^{-1}(X) = b_1^{0}\Big|\sum_{k=1}^{m}\big(d_{k}^1-2\Delta t~ d_{1(k)}\big)\nabla\varphi_{k}(y)\Big|_0^2,
	$$
	considering the approximation $b_1^{-1}\approx b_1^{0}$.

To determine the solution vector $d^\eta=(d_1^\eta,d_2^\eta, \cdots, d_m^\eta)$, we must now solve the algebraic system \eqref{sistema_n=0} and \eqref{sistema_n>0}, for $\eta=0,1,\cdots, N$. As the system is not linear, we will employ the Newton Method, which we briefly describe below, aiming at an optimization of the computational cost in the calculations.

	\subsection{Newton's method}
	
For each  $\eta=1,2,\cdots, N$, consider $\mathcal{F}(X)$ the left side of the equation in \eqref{sistema_n=0}.
In this way, solve \eqref{sistema_n=0} is equivalent to finding the root of $\mathcal{F}(X)$.	By Newton's method, given the initial approximation and denoting $X=d^{n+1}$,, we obtain $X_{k+1}$ a new approximate solution of  $\mathcal{F}(X)=0$, given by $X_{k+1}=X_k+s_k$, where $s_k$ is the linear system solution  $J\mathcal{F}(X_k)s_k=-\mathcal{F}(X_k)$.
The key point of the method is the calculation of the Jacobian matrix $J\mathcal{F}(X)$. This is because, if it has a high computational cost, it can make the method unfeasible.

Define

\begin{equation}\label{newton_n>0}
			\mathcal{F}^{\eta}(X) = \Big(M_1^{\eta+1} + \theta(\Delta t)^2 G^{\eta+1}(X)K_1\Big)X + \Gamma^\eta,
\end{equation}
		where $~\Gamma^\eta = M_2^{\eta}d^\eta + \Big(M_3^{\eta-1} + \theta~\!(\Delta t)^2 G^{\eta-1}(d)K_1\Big) d^{\eta-1}
- (\Delta t)^2 F^{\eta+\theta}$.

For $\eta = 0$, follows
\begin{equation}\label{newton_n=0}
			\mathcal{F}^{0}(X) = \Big[M_1^{1} + M_3^{0} + \theta(\Delta t)^2 \big(G^{1}(X)+G^{-1}(X)\big)K_1\Big]X + 2\theta(\Delta t)^3 G^{-1}(X)K_1 d_1 + \Gamma^0,
\end{equation}
		where $~\Gamma^0 = M_2^{0}d^0 - 2\Delta t~\!M_3^{0}d_1
- (\Delta t)^2 F^{0+\theta}$.

For 	$\eta > 0$:
\begin{equation}\label{jacobiana_n>0}
		\begin{aligned}
			J{\mathcal{F}}^\eta(X) = \frac{d~\!\mathcal{F}^\eta\!\!}{dX}(X) &= M_1^{\eta+1} + \theta(\Delta t)^2 \frac{d}{dX}\Big[G^{\eta+1}(X)K_1 X\Big]\\
			&= M_1^{\eta+1} + \theta(\Delta t)^2\Big(K_1 X~\!\frac{d G^{\eta+1}}{d X}(X) + G^{\eta+1}(X)K_1\Big),
		\end{aligned}
\end{equation}
	and for $\eta = 0$:
\begin{equation}\label{jacobiana_n=0}
		\begin{aligned}
			J{\mathcal{F}}^0(X) &= M_1^{1} + M_3^{0} +  \theta(\Delta t)^2 \Big[K_1 X~\!\Big(\frac{d G^{1}}{d X}(X)+\frac{d G^{-1}}{d X}(X)\Big) + \Big(G^{1}(X)+G^{-1}(X)\Big)K_1\Big]\\
			&\quad + 2\theta(\Delta t)^3~\!K_1 d_1~\!\frac{d G^{-1}}{d X}(X),
		\end{aligned}
\end{equation}
	where
	\begin{align*}	
		\frac{d G^{\eta}}{d X}(X) = \Big[\frac{\partial G^{\eta}}{\partial X_1}(X)~~ \cdots~~ \frac{\partial G^{\eta}}{\partial X_m}(X)\Big]_{m\times 1},\quad\forall \eta\in\{-1\}\cup\{1,\cdots,N-1\}.
	\end{align*}
	Using \eqref{def:G_Ls}, we obtain
	\begin{align*}
		&\frac{\partial G^{\eta}}{\partial X_k}(X)  = 2~\!b_1^{\eta}~X_k~\Big[\int_{\Omega}\big(\nabla_{y_i}\varphi_k(y)\big)^2~dx\Big],\\
		&\frac{\partial G^{-1}}{\partial X_k}(X)    = 2~\!b_1^{0}~\big(X_k - 2\Delta t~\!d_{1k}\big)~\Big[\int_{\Omega}\big(\nabla_{y_i}\varphi_k(y)\big)^2~dx\Big],
	\end{align*}
	where $d_{1k}$ is the $k$-th coordinate of the initial velocity vector $d_1$.
	
	From there, we apply Newton's refinement algorithm, for $\eta= 1,2, \ldots, N-1 $, until one of the stopping criteria is satisfied: $\Vert X_{k+1} - X_k \Vert_\infty <10^{-14}$ or $ \vert \mathcal{F}(X_k)\vert<10^{-14}$, where $X_k $ is the solution obtained in the $i$-th iteration of Newton. We also define a maximum limit for iterations in order to avoid infinite refinements.

\subsection{Finite Element Method- Hermite Base}
	
Let $[T_h]$ be a polygonalization family of $\Omega$, where $T_h =\{K; ~ K \subset\Omega\}$, satisfying
the minimum angle condition (see \cite{ciarlet}) and indexed by the parameter $h$, representing the
maximum diameter of the elements $K\in T_h$. Given an integer $l\geq 1$, we introduce the following finite element space	 $V_h$
spanned by the Hermite interpolation polynomials basis functions. More specifically, we let
\begin{align*}
		V_h = \Big\{\psi_h\in V_h^3 \cap C^1(\Omega);~\psi_h\big|_K \in P_l(K),~\forall~K\in T_h\Big\}\subset V,
	\end{align*}
where  $V^3_{h}$ represents a finite-dimensional function space generated by piecewise polynomial $P_l(K)$ functions of degree less or equal $3$  associated with a discretization
by finite elements  of size $h$. In particular, in this paper we will use the Hermite polynomials of degree equal $l=3$.
	
	From the Lemma of Douglas-Dupont (see in \cite{ciarlet}), it follows that, given a function
	$w:]0,T[\longrightarrow H^{l+1}(\Omega)$, there is an interpolator $\widehat{w}_h: ]0,T[\longrightarrow V_h^l$ such that
\begin{align*}
	\Vert w(t) - \widehat{w}_h(t)\Vert_m \leq C_1~\!h^{l+1-m}\Vert w(t)\Vert_{l+1},
\end{align*}
	where $\Vert\cdot\Vert_m$ denotes the seminorm over Hilbert space $H^m(\Omega)$ and $0\leq m\leq l$.

Knowing the Hermite polynomials, then all the matrices of the system \eqref{sistema_n>0} can be calculated explicitly and solving the nonlinear system  by Newton's method gives the approximate solution, given by
$$
v_h(y,t) = \sum_{i=1}^m d_i(t)\varphi_i(y), \quad\varphi_i(y)\in V_h.
$$

We note that since the Hermite polynomials are orthogonal then $v^n_h(y_i)=d_i^n$

\section{Numerical analysis}\label{sec:analNum}
In this section we present error estimates for the finite element semi-discrete and fully discrete approximations of the nonlinear beam equation \eqref{prob:fixaVar} or  equivalently for the weak formulation of the  nonlinear beam equation with moving ends.

\textbf{Remarks on regularity:} In order to get an error estimate,  we  require more  regularity for the function $v=v(y,t)$ and its derivatives in time:
\begin{equation}\label{hip:regularidades_v's_H2}
			v', v''  \in L^{2}(0,T;H^4(\Omega)),~~
			v'''     \in L^{\infty}(0,T;H^1(\Omega))~~\mbox{and}~~
			v^{(iv)} \in L^{\infty}(0,T;L^2(\Omega)).
			\tag{H2}
\end{equation}

	\subsection{Error estimates for the semi-discrete problem}
	
The semi-discrete problem consists of finding $v_h:~[0,T[ \longrightarrow V_h$ such that satisfies the problem \eqref{prob:fixaVar} restricted to $w=w_h \in V_h$, that is, to find $v_h(t)$ for all $t\in~\!] 0, T [$ such that	
		\begin{equation}\label{prob:fixaVarAprox}
		\left\{
		\begin{aligned}
			&\begin{aligned}
				&\big(v_h''(t),w_h\big) + b_1(t)\big|\nabla v_h(t)\big|_0^2 \big(\nabla v_h(t),\nabla w_h\big) + b_2(t)\big(\Delta v_h(t),\Delta w_h\big) + \nu~\!\big(v_h'(t),w_h\big) \\
				&\quad+ \big(a^{(1)}_i(t)\nabla_{y_i} v_h(t),\nabla_{y_i}w_h\big) + \big(a^{(2)}_{ij}(t)\nabla_{y_i,y_j}v_h(t),w_h\big) \\
				&\quad+ \big(a^{(4)}_i(t)\nabla_{y_i}v_h'(t),w_h\big) + \big(a^{(5)}_i(t)\nabla_{y_i}v_h(t),w_h\big) = 0, \mbox{ in } Q,
			\end{aligned}\\		
			&~\big(v_h(y,0),w_h\big) = \big(v_{0h}(y),w_h\big),~~ \big(v_h'(y,0),w_h\big) = \big(v_{1h}(y),w_h\big) \quad\mbox{in } ~\Omega,\\
			&~\big(v_h,w_h\big) = 0~\mbox{and}~ \big(\nabla v_h,w_h\big) \equiv 0,~\mbox{on}~ \partial\Omega\times[0,T[.
		\end{aligned}
		\right.
    \end{equation}
For the estimates, we will need the following hypothesis about the projections of $V$ in $V_h$, denoted by $R_h$ and defined in \eqref{def:projecaoOrtogonal}, of the initial data
 \begin{equation}\label{hip:condInicialContinuo_H3}
		\vert v_{1h} - R_h v_1\vert_0~\leq~c_1h^4,\quad\vert\nabla v_{0h} - R_h \nabla v_0\vert_0~\leq~c_2h^3\quad\mbox{and}\quad \vert\Delta v_{0h} - R_h \Delta v_0\vert_0~\leq~c_3h^2,
		\tag{H3}
    \end{equation}
where  $R_h\Delta v_0,~R_h \nabla v_0$ and $R_h v_1$ are the projections of $\Delta v_{0},~\nabla v_{0}$ and  $v_{1}$ in $V_h$, respectively, and $c_1,c_2,c_3\geq 0$ are constants independent of  $h$.

Let us assume that,
\begin{align*}\label{hip:zeta0_H4}\tag{H4}
    		 \max_{t\in[0,T]}(K'(t))^2 < \frac{\zeta_0}{4} 
		\end{align*}
which requires that the speed of the end points be smaller than the characteristic speed of the equation.		
		
   Let us denote by $e(t) = v(t) - v_h(t), \forall t\in ]0,T[$, the error between the exact and approximate solution. Then, we have the following error estimate:

    \begin{myth}(Semi-discrete problem - Estimate)
    	\label{teo:tempoContinuo}
    	
    	Let $v=v(y,t)$ the solution of the problem \eqref{prob:fixaVar}. Under hypotheses
    	\ref{hip:dominio2D_H1}-- \eqref{hip:zeta0_H4} and given the initial data
    	 $v_0,v_1\in H_0^2(\Omega)\cap H^4(\Omega)$, there exists a constant $C>0$, such that
    	\begin{equation}\label{estimativaContinuo}
    		\Vert e'\Vert_{L^\infty(0,T;L^2(\Omega))} + \Vert e\Vert_{L^\infty(0,T;H^2_0(\Omega))} ~\leq~ C h^2.
    	\end{equation}
    \end{myth}
 where the  constant $C$ is dependent on $v$ and independent of $h$. 
    \begin{proof}


  We note that, as $V_h\subset V=H_0^2(\Omega)$, taking in particular $w=w_h$ in \eqref{prob:fixaVar}$_1$ and doing the difference with \eqref{prob:fixaVarAprox}$_1$, denoting $e = e(t)$, $v = v(t)$ and
  $v_h=v_h(t) $ , we obtain:
\begin{equation}\label{estCont:dif_v_vh}
			\begin{aligned}
				&\big(e'',w_h\big) - b_1(t)\big(\big[\big|\nabla v\big|_0^2\nabla v- \big|\nabla v_h\big|_0^2\nabla v_h\big],\nabla w_h\big) + b_2(t)\big(\Delta{e},\Delta w_h\big) \\
				&\quad+ \nu~\!\big(e',w_h\big) + \big(a^{(1)}_i(t)\nabla_{y_i} e~,\nabla_{y_i}w_h\big) + \big(a^{(2)}_{ij}(t)\nabla_{y_i,y_j}e~,w_h\big) \\
				&\quad+ \big(a^{(4)}_i(t)\nabla_{y_i}e',w_h\big) + \big(a^{(5)}_i(t)\nabla_{y_i}e~,w_h\big) = 0,~ \forall~w_h\in V_h.
			\end{aligned}
\end{equation}	

We use the  standard decomposition of the error
\begin{equation}\label{estCont:decompErro}
e(t) = v(t) - v_{h}(t) = (v(t) - R_hv(t)) + (R_hv(t) - v_{h}(t))=\rho(t)+\xi(t)= \rho+\xi
\end{equation}
where $\rho =  v(t) - R_h v(t)$ and  $\xi= R_h v(t) - v_h(t)$,  with $R_h v(t)$  being the Ritz projection of $v(t)\in V_h$, to be defined as follows:

To define the Ritz projection $R_h$ we will follow the analogue of the ideas of the elliptic approximation theory. Consider the following definitions.	
\begin{mydef}(Bilinear form)
	  		
	        Let the functional $a^{(\vartheta,t)}:V\times V\rightarrow\mathbb{R}$, for all $\vartheta \in \{v,v_h\}$ and $t\in [0,T[$ both fixed,  such that
		\begin{equation}\label{def:formaBilinear}
	    		a^{(\vartheta,t)}(v,w) = (a^{(1)}_i(t)\nabla_{y_i} v~,\nabla_{y_i}w) + b_1(t)\vert\nabla \vartheta(t)\vert_0^2(\nabla v,\nabla w) + b_2(t)(\Delta{v},\Delta{w}).
	   		 \end{equation}			
		\end{mydef}
	
	    \begin{mydef}(Ritz Projection on $a^{(\vartheta,t)}(\cdot,\cdot)$)	
We define the Ritz projection on the functional $a^{(\vartheta,t)}(\cdot,\cdot)$, defined in \eqref{def:formaBilinear}, as follows:		
  	\begin{align*}
        		R_h:~ V &\longrightarrow V_h \\
          	          v &\longmapsto     R_h v
        	\end{align*}
	       satisfying the orthogonality condition
    	    \begin{equation}
        	    \label{def:projecaoOrtogonal}
	            a^{(\vartheta,t)}\big(v(t)-R_h v(t),w_h\big) = 0,~ \forall w_h\in V_h.
    	    \end{equation}
    	\end{mydef}
   For the Ritz projection to satisfy the equation \eqref{def:projecaoOrtogonal}, it is necessary that the functional \eqref{def:formaBilinear} satisfies the following Lemma:
      	\begin{mylem}(Functional Bilinear, Continuous and Coercive)
    	
			The functional $a^{(\vartheta,t)}(\cdot,\cdot)$, defined by \eqref{def:formaBilinear}, is bilinear, continuous and coercive, for all $t\in[0,T[$ and $\vartheta\in\{v,v_h\}$.
    	\end{mylem}
    	
    	\begin{proof} The proof is given in three steps:

\begin{enumerate}
			
	\item The bilinearity of the functional $a^{(\vartheta,t)}(\cdot,\cdot)$ is immediate.
			
	\item The coercivity of the functional $a^{(\vartheta,t)}(\cdot,\cdot)$:
	\begin{equation}\label{formaBilinear:coerciva}
	   	\begin{aligned}
			a^{(\vartheta,t)}(w,w) &= \big(a^{(1)}_i(t),(\nabla_{y_i}w)^2\big) + b_1(t)\vert\vartheta(t)\vert_0^2~\!\vert\nabla w\vert_0^2 + b_2(t)\vert\Delta{w}\vert_0^2 \\
								    &\geq \kappa_0~\!\big[\vert\nabla_{y_i}w\vert_0^2 + \vert\nabla w\vert_0^2 + \vert\Delta{w}\vert_0^2\big]\geq c~\!\Vert w\Vert_V^2.
		\end{aligned}
	\end{equation}
	The constants are given by
	\begin{align*}
		\kappa_0 = \min\big\{\min_{t\in[0,T]} b_1(t)\vert\vartheta(t)\vert_0^2,~\min_{t\in[0,T]} b_2(t),~\min_{i\in\{1,\cdots,n\}}\big[\min_{t\in[0,T]}\big(\inf_{y\in \Omega}a^{(1)}_i(y,t)\big)\big]\big\}.
	\end{align*}			

	We note that, from \eqref{def:asEbs} and $\vert\vartheta\vert_0^2\not\equiv 0$, once $\vartheta\in\{v,v_h\}$ and $v,v_h$ are non-trivial solutions,
	\begin{equation}\label{formaBilinear:b2>0}
		\begin{aligned}		
			\min_{t\in[0,T]} b_1(t)\vert\vartheta\vert_0^2 &= \min_{t\in[0,T]} \frac{\zeta_1\vert\vartheta\vert_0^2}{K^4} > \frac{\zeta_1\vert\vartheta\vert_0^2}{K_1^4} > 0\\
			\min_{t\in[0,T]} b_2(t) &= \min_{t\in[0,T]} ~~\frac{1}{K^4}~~\! > ~~\!\frac{1}{K_1^4}~~\! > 0~\!.
		\end{aligned}
	\end{equation}		
	And from hypothesis \eqref{hip:zeta0_H4},
	\begin{equation}\label{formaBilinear:a_i^(1)>0}	
		\begin{aligned}
			\min_{t\in[0,T]}\big(\inf_{y\in \Omega}a^{(1)}_i(y,t)\big) &= \min_{t\in[0,T]}\Big\{ \frac{1}{K^2}\big[\zeta _0 - 4(K')^2\big]\Big\} > 0.
		\end{aligned}							
	\end{equation}		
	Therefore, by \eqref{formaBilinear:b2>0} and \eqref{formaBilinear:a_i^(1)>0}, $\kappa_0>0$, concluding the coercivity.			
	
	\item The continuity of the functional $a^{(\vartheta,t)}(\cdot,\cdot)$:
	\begin{align}
						   		   \nonumber
		a^{(\vartheta,t)}(v,v) &\leq \vert(a^{(1)}_i(t)\nabla_{y_i} v~,\nabla_{y_i}w)\vert + b_1(t)\vert\nabla \vartheta(t)\vert_0^2\vert(\nabla v,\nabla w)\vert + b_2(t)\vert(\Delta{v},\Delta{w})\vert\\
						   		   \label{formaBilinear:continua}
					        	   &\leq \kappa_1\big[|\nabla_{y_i} v|_0~|\nabla_{y_i} w|_0 + \vert\nabla v\vert_0\vert\nabla w\vert_0^2 + |\Delta v|_0~|\Delta w|_0\big] \\
					        	  \nonumber
					        	   &\leq~ c~\!||v||_V||w||_V,
	\end{align}
	where we consider $\vert\nabla_{y_i}v\vert_0 \leq \vert\nabla v\vert_0$ and apply the inequalities of Cauchy-Schwarz and Poincare-Friedrich and the equivalence of norms $ \vert \Delta v\vert_0 $ and $\Vert v\Vert_2$ in $H_0^2(\Omega)$. The constant are given by
	\begin{align*}
		\kappa_1 = \max_{i\in\{1,\cdots,n\}}\big\{\max_{t\in[0,T]}\big[\sup_{y~\!\in~\!\Omega} |a_i^{(1)}(y,t)| + b_1(t)(\vert v(t)\vert_0^2 + \vert v_h(t)\vert_0^2) +  b_2(t)\big]\big\}.
	\end{align*}
	Once $\{v(t),v_h(t)\}\subset V=H_0^2(\Omega)$, we have $\vert v\vert_0 + \vert v_h\vert_0^2 < +\infty$.
\end{enumerate}				
	
	Thus, we conclude that the functional $a^{(\vartheta,t)}(\cdot,\cdot)$,  is bilinear, continuous and coercive, $\forall t\in[0,T[$ and $\vartheta\in\{v,v_h\}$.		
\end{proof}			

	We return the proof of the theorem: Summing and subtracting the Ritz projection of $v$ and its derivatives conveniently on each term of the equation \eqref{estCont:dif_v_vh}, using the decomposition of the error \eqref{estCont:decompErro}, and taking in particular $w_h=\xi'\in V_h$ and using the \eqref{def:projecaoOrtogonal}, we get:		
	\begin{equation}\label{estCont:somaSubProj}
		\Gamma_1(\xi,v_h,t) + \Gamma_2(\xi,t) = -\big[\Gamma_1(\rho,v,t) + \Gamma_2(\rho,t) + \mu(t)\big],
	\end{equation}
	where $\Gamma_1: V\times\{v,v_h\}\times[0,T[~\longrightarrow~\mathbb{R}$, $~\Gamma_2: V\times[0,T[~\longrightarrow~\mathbb{R}~$ and $~\mu:[0,T[~\longrightarrow~\mathbb{R}~$ such that
	\begin{align*}	
		\Gamma_1(\chi,\vartheta,t) &= \big(\chi'',\xi'\big) + a^{(\vartheta,t)}(\chi,\xi') + \nu~\!\big(\chi',\xi'\big) +  \big(a^{(4)}_i(t)\nabla_{y_i}\chi',\xi'\big),\\
		\Gamma_2(\chi,t) 			&= \big(a^{(2)}_{ij}(t)\nabla_{y_i,y_j}\chi,\xi'\big) + \big(a^{(5)}_i(t)\nabla_{y_i}\chi,\xi'\big),\\
		\mu(t) &= b_1(t)\big[\vert\nabla v_h\vert_0^2 - \vert\nabla v\vert_0^2\big]\big(R_h \nabla v,\nabla\xi'\big).
	\end{align*}

	Next, we will analyze each term of the equation \eqref{estCont:somaSubProj} \\

 \noindent\textbf{Analisys: Terms dependent of $\xi$}
 
 	By definition, we have
 	\begin{align*}
 		\Gamma_1(\xi,v_h,t) = \big(\xi'',\xi'\big) + a^{(v_h,t)}(\xi,\xi') + \nu~\!\big(\xi',\xi'\big) +  \big(a^{(4)}_i(t)\nabla_{y_i}\xi',\xi'\big).
 	\end{align*}
 	Analysing each term, we have
	\begin{equation}\label{estCont1}		
		\big(\xi'',\xi'\big) = \frac{1}{2}\frac{d}{dt}|\xi'|_0^2,
		\qquad \nu\big(\xi',\xi'\big) = \nu~\!|\xi'|_0^2 \geq 0.
	\end{equation}	
	Integrating by parts and using the definition \eqref{def:asEbs}, follows that
  	\begin{equation}\label{estCont2}
			\big(a^{(4)}_i(t)\nabla_{y_i}\xi',\xi'\big) = - \frac{1}{2}\int_\Omega \big(\nabla_{y_i} a^{(4)}_i(y,t)\big)\big(\xi'\big)^{2} dy = \frac{K'(t)}{K(t)}\vert\xi'\vert_0^2\geq 0,
	   	\end{equation}
since the $\ds\nabla_{y_i} a^{(4)}_i(y,t) = -\frac{2K'(t)}{K(t)}$ and $K'(t) > 0$, for the hypotesis \eqref{hip:dominio2D_H1}. 

	Now using the definition \eqref{def:asEbs} and the mean value theorem for integrals, we obtain
	\begin{equation}\label{estCont3}
	   	\begin{aligned}
	   		a^{(v_h,t)}(\xi,\xi') &= \frac{1}{2}~\Big[\int_{\Omega}a^{(1)}_i(t)\frac{d}{dt}(\nabla_{y_i} \xi)^2~dy\Big] + \frac{b_1(t)}{2}\vert v_h\vert_0^2~\!\frac{d}{dt}\vert\nabla\xi\vert_0^2 + \frac{b_2(t)}{2}\frac{d}{d t}|\Delta\xi|_0^{2}\\
	   							   &= \frac{1}{2}\frac{d}{d t}\big[a^{(1)}_i(\sigma_1,t)~\!\vert\nabla_{y_i}\xi\vert_0^2 + b_1(t)\vert v_h\vert_0^2\vert\nabla\xi\vert_0^2 + b_2(t)|\Delta\xi|_0^{2}\big]\\
	   							   &\qquad-\frac{1}{2}~\big[a_i'~\!^{(1)}(\sigma_2,t)~\!\big|\nabla_{y_i} \xi\big|_0^2 + \frac{d}{dt}\big(b_1(t)\vert v_h\vert_0^2\big)\vert\nabla\xi\vert_0^{2} + b'_2(t)|\Delta\xi|_0^{2}\big],
		\end{aligned}
 	\end{equation}
 	where $\sigma_1,\sigma_2\in\Omega$. So, by \eqref{estCont1} -- \eqref{estCont3} and decreasing the term $\ds\frac{\kappa_2}{2}\vert\xi'\vert_0^2$ :
 	\begin{equation}\label{estCont_Gamma1}
 		\begin{aligned}
	 		\Gamma_1(\xi,v_h,t) &\geq \frac{1}{2}\frac{d}{dt}\big[|\xi'|_0^2 + a^{(1)}_i(\sigma_1,t)~\!\vert\nabla_{y_i}\xi\vert_0^2 + b_1(t)\vert v_h\vert_0^2\vert\nabla\xi\vert_0^2 + b_2(t)|\Delta\xi|_0^{2}\big] - \frac{\kappa_2}{2}E(t),
 		\end{aligned}
 	\end{equation}
 	where
 	\begin{align*}
 		&E(t) = \vert\nabla_{y_i}\xi\vert_0^{2} + \vert\nabla\xi\vert_0^{2} + \vert\Delta\xi\vert_0^2 + \vert\xi'\vert_0^2,
 	\end{align*}
 	and $\ds\kappa_2 = \max_{i,j\in\{1,\cdots,n\}}\big\{\max_{t\in[0,T]}\big[\sup_{y~\!\in~\!\Omega} \big(|a_{i}'~\!^{(1)}(y,t)| + \Big|\frac{d}{dt}\big(b_1(t)\vert v_h(t)\vert_0^2\big)\Big| + |b'_2(t)|\big)\big]\big\}$.
 	
 	We note that, by definition \eqref{def:asEbs} and applying the inequality of Cauchy-Schwarz,
 	\begin{equation}\label{analysing}
	 	\begin{aligned}
	 		\Big|\frac{d}{dt}\big(b_1(t)\vert v_h\vert_0^2\big)\Big| &= b'_1(t)\vert v_h\vert_0^2 + b_1(t)\frac{d}{dt}\big(v_h,v_h\big) \leq c\big[\vert v_h\vert_0^2 + \vert v_h\vert_0\vert v'_h\vert_0\big] < +\infty
	 	\end{aligned}
  	\end{equation}
 	where $\ds c = \frac{4~\!\zeta_1}{K_0^5}(K_2 + K_0)$.
 	By Theorem  \eqref{teo:EeUfixa2D}, we have $v'_h(t)\in L^2(\Omega)$. Soon, by \eqref{analysing}, we ensure $\kappa_2$ limited.
 	
 	Now, analyzing each term of $\Gamma_2(\xi,t)$, give by
 	\begin{align*}
 		\Gamma_2(\xi,t) = \big(a^{(2)}_{ij}(t)\nabla_{y_i,y_j}\xi,\xi'\big) + \big(a^{(5)}_i(t)\nabla_{y_i}\xi,\xi'\big),
	\end{align*}
	we have, using the inequality of Cauchy-Schwarz and Young, the definition of $\vert\Delta\xi\vert_0$ and increasing the norm $\vert\nabla\xi\vert_0^{2}$, follows that
 	\begin{equation}\label{estCont_Gamma2}
 		\begin{aligned}
 			\Gamma_2(\xi,t) \leq \frac{\kappa_3}{2}E(t),
 		\end{aligned}
 	\end{equation}
 	where $\kappa_3 = \ds 2\!\!\!\max_{i,j\in\{1,\cdots,n\}}\big\{\max_{t\in[0,T]}\big[\sup_{y~\!\in~\!\Omega} \big(|a_{ij}^{(2)}(y,t)| + |a_{i}^{(4)}(y,t)| + |a_{i}^{(5)}(y,t)|\big)\big]\big\}$.\\
 	
	\noindent\textbf{Analisys: Terms dependent on $\rho$}
   	
	By definition, we have
	\begin{align*}
		\Gamma_1(\rho,v,t) + \Gamma_2(\rho,t) &= \big(\rho'',\xi'\big) + a^{(v,t)}(\rho,\xi') + \nu~\!\big(\rho',\xi'\big) +  \big(a^{(4)}_i(t)\nabla_{y_i}\rho',\xi'\big) \\
		&\qquad+ \big(a^{(2)}_{ij}(t)\nabla_{y_i,y_j}\rho,\xi'\big)+ \big(a^{(5)}_i(t)\nabla_{y_i}\rho,\xi'\big).
	\end{align*}	   	   		        
   	Analysing the terms of $\Gamma_1(\rho,v,t)$ and $\Gamma_2(\rho,t)$ applying the inequality of Cauchy-Schwarz and Young, follows that
%
	\begin{equation}\label{estCont_Upsilon_rho}
		\begin{aligned}
			\Gamma_1(\rho,v,t) + \Gamma_2(\rho,t) &\leq \frac{\kappa_4}{2}\big[\vert\rho''\vert_0^2 + \vert\rho'\vert_0^2 + \vert\nabla\rho'\vert_0^2 + \vert\nabla\rho\vert_0^2 + \vert\Delta\rho\vert_0^2 + \vert\xi'\vert_0^2\big] \\
			&\leq \frac{\kappa_4}{2}\big[\vert\rho''\vert_0^2 + \Vert\rho'\Vert_1^2 + \Vert\rho\Vert_2^2 + \vert\xi'\vert_0^2\big],
		\end{aligned}
	\end{equation}
	once $\vert\rho'\vert_0^2 + \vert\nabla\rho'\vert_0^2 \leq \Vert\rho'\Vert_1^2$ and $\vert\nabla\rho\vert_0^2 + \vert\Delta\rho\vert_0^2 \leq \Vert\rho\Vert_2^2$, and where 
	$$
	\kappa_4 = 5(\kappa_1 + \kappa_2 + \kappa_3 + 1 + \nu).
	$$

	We note that by definition of Ritz projection in \eqref{def:projecaoOrtogonal}, we get $a^{(v,t)}(\rho,\xi') = a^{(v,t)}(v - R_h v,\xi') = 0$.\\

	\noindent\textbf{Analisys: Nonlinear term}
		
	For the term $\mu(t)$, after integrating by parts, applying the module and the inequalities of Cauchy-Schwarz and Young, we get
	\begin{equation}\label{estCont_mu}
		\begin{aligned}
	 		&\mu(t) 
	 		\leq b_1(t)~\!\vert\nabla v - \nabla v_h\vert_0\big[\vert\nabla v\vert_0 + \vert\nabla v_h\vert_0\big]\big|\big(\nabla R_h \nabla v, \xi'\big)\big| \\
 			&\qquad\!\leq~ \kappa_1\big[\vert\nabla\rho\vert_0 + \vert\nabla\xi\vert_0\big]~\!\vert \nabla R_h \nabla v\vert_0\vert\xi'\vert_0 ~\leq \frac{\kappa_5}{2}\big[\vert\nabla\rho\vert_0^2 + \vert\nabla\xi\vert_0^2 + \vert\xi'\vert_0^2\big],
 		\end{aligned}
	\end{equation}
	where $\kappa_5 = \kappa_4 + 2~\!\kappa_1~\!\vert \nabla R_h\nabla v\vert_0 < +\infty$, once $R_h \nabla v\in V_h\subset V=H_0^2(\Omega)$ implies $\nabla R_h \nabla v\in L^2(\Omega)$.\\

	\noindent\textbf{Final estimate}
	
	Now, passing the term $\Gamma_2(\xi,t)$ to the right side of \eqref{estCont:somaSubProj} and using the estimates \eqref{estCont_Gamma1} -- \eqref{estCont_mu}, we obtain
	\begin{equation}\label{estCont:preintegral}
		\frac{1}{2}\Big[\frac{d}{dt}\Theta_1(t)-\kappa_2~\!E(t)\Big] \leq \Gamma_1(\xi,v_h,t) = -\big[ \Gamma_1(\rho,v,t) + \Gamma_2(\rho+\xi,t) + \mu(t)\big] \leq \frac{\kappa_5}{2}~\!\Theta_3(t)
	\end{equation}
	where
	\begin{align*}		
		&\Theta_1(t) = |\xi'|_0^2 + a^{(1)}_i(\sigma_1,t)~\!\vert\nabla_{y_i}\xi\vert_0^2 + b_1(t)\vert v_h\vert_0^2\vert\nabla\xi\vert_0^2 + b_2(t)|\Delta\xi|_0^{2},\\
		&\Theta_2(t) = E(t) + \vert\rho''\vert_0^2 + \Vert\rho'\Vert_1^2 + \Vert\rho\Vert_2^2.
	\end{align*}	
		
	Passing the term $\kappa_2~\!E(t)$ of left to the right side of \eqref{estCont:preintegral}, integrating from $0$ to $t<T$ and multiplying by both side by $2$, we get
	\begin{equation}\label{estCont:integral}
		\kappa_0~\!E(t) \leq \kappa_1~\!E(0) + \kappa_5\Big[\int_0^t \vert\rho''\vert_0^2 + \Vert\rho'\Vert_1^2 + \Vert\rho\Vert_2^2 ~dt + \int_0^t E(s)~dt\Big].
	\end{equation}
		
	From hypothesis \eqref{hip:condInicialContinuo_H3}, follows that
	\begin{equation}\label{estCont:condInicial}
		E(0) \leq C_1h^4, \quad \mbox{for some}\quad C_1>0.
	\end{equation}
	
	On the other hand, by the C\'{e}a Lemma and Douglas-Dupont Theorem, follows that
	\begin{equation}\label{estCont:normaRho}
		\begin{aligned}
			&\int_0^t \vert\rho''(s)\vert_0^2 + \Vert\rho'(s)\Vert_1^2 + \Vert\rho(s)\Vert_2^2~ds \leq~ C_2~\!h^4\sum_{k=0}^{2}\Big\Vert\frac{d^k v}{dt^k}\Big\Vert^2_{L^2(0,T;H^4(\Omega))}, ~\mbox{for some}~C_2>0.
		\end{aligned}
	\end{equation}
	Substituting the estimations \eqref{estCont:condInicial} and \eqref{estCont:normaRho} in \eqref{estCont:integral} and dividing both sides by $\kappa_0$, since $\kappa_0 > 0 $ by \eqref{formaBilinear:b2>0} and \eqref{formaBilinear:a_i^(1)>0}, we obtain
	\begin{equation}\label{estCont:gronwall}
		E(t) \leq C_3~\!h^4 + C_3 \int_0^t E(s)~ds,
	\end{equation}		
	where $~C_3 = (\kappa_5/\kappa_0)(C_1+C_2)~\!$.			
	As the estimate \eqref{estCont:gronwall} holds for all $t\in[0, T [$, for Gronwall-Bellman's Lemma, we have to		
	\begin{equation}\label{estCont:xi_final}
		\frac{1}{2}\big(\vert\xi'(t)\vert_0 + \vert\Delta\xi(t)\vert_0\big)^2\leq \vert\xi'(t)\vert_0^2 + \vert\Delta\xi(t)\vert_0^2 \leq E(t) \leq C_5~\!h^4,~\mbox{where}~C_5 = C_3e^{C_3 T}.
	\end{equation}		 			
	Therefore, from the decomposition of the error \eqref{estCont:decompErro} and the estimates \eqref{estCont:normaRho} and \eqref{estCont:xi_final}, we obtain	
	\begin{equation}\label{estCont:final}
		\vert e'(t)\vert_0 + \vert \Delta e(t)\vert _0 \leq \vert \rho'(t)\vert _0 + \vert \Delta\rho(t)\vert _0 + \vert \xi'(t)\vert _0 + \vert \Delta\xi(t)\vert _0 \leq C~\!h^2,~\forall t\in[0,T[,
	\end{equation}
	for some $C>0$ independent of $t$ and $h$.		

	Finally, we just take \eqref{estCont:final} the  essential supreme in time and using the equivalence of norms $\vert\Delta e(t)\vert_0$ and $\Vert e(t)\Vert_2$ in $H_0^2(\Omega)$ to obtain the estimate \eqref{estimativaContinuo}.
    \end{proof}

 \subsection{Fully-Discrete Problem}	
   	
   	In order to obtain an error estimate in the norm $L^\infty(0,T);L^2(\Omega)$ for the  fully-discrete problem, we discretize the fixed  time interval $[0,T]$, using the discretization defined in \eqref{sistemaNewmark}.

   Let $\theta \in [0,1]$ and consider the notation
	 $v^\eta = v(\cdot,t_\eta)$, for  $\eta\in\{1,\cdots,N-1\}$ and the operators differences
	\begin{equation}\label{def:difsFinitas}
		\begin{aligned}
			&\delta^2 v^\eta = \frac{v^{\eta+1} -2~\!v^\eta + v^{\eta-1}}{(\Delta t)^2}, \quad\delta v^\eta = \frac{v^{\eta+1} - v^{\eta-1}}{2~\!\Delta t},\quad \delta v^{\eta+1/2} = \frac{v^{\eta+1} - v^{\eta}}{\Delta t},\\
			&\qquad\delta v^{\eta-1/2} = \frac{v^{\eta} - v^{\eta-1}}{\Delta t},\quad v^{\eta+\theta} = \theta(v^{\eta+1}+v^{\eta-1}) + (1-2\theta)v^\eta.
		\end{aligned}
	\end{equation}
Using the operators \eqref{def:difsFinitas}	, we introduce the following fully discrete method based on the semi-discrete problem \eqref{prob:fixaVar}, given by the following equation:
		\begin{equation}\label{prob:fixaVarDiscretoExato}
		\begin{aligned}
			&\big(\delta^2 v^\eta,w\big) + a^{\eta+\theta}(v,w) + \nu~\!\big(\delta v^\eta ,w\big) + \big(\big[a^{(2)}_{ij}\nabla_{y_i,y_j}~\!v\big]^{\eta+\theta},w\big) \\
			&\quad+ \big(a^{(4,\eta)}_i\nabla_{y_i}\delta v^\eta,w\big) + \big(\big[a^{(5)}_i\nabla_{y_i}v\big]^{\eta+\theta},w\big) = \big(\psi^\eta,w\big),~\forall~w\in V,
		\end{aligned}
	\end{equation}
	where $a_i^{(4,\eta)} = a_i^{(4)}(y,t_\eta)~$ and
	\begin{equation}\label{def:formaBilinearDisc}
		\begin{aligned}	
			&a^{(\vartheta,~\!\eta+\theta)}(v,w) = \sum_{k=1}^{3}\big(\big[f_k D_k v\big]^{\eta+\theta}, D_k w\big),\\
			&\psi^\eta = \delta^2 v^\eta - v_{tt}^{\eta+\theta} + \nu~\!\big(\delta v^\eta - v_t^{\eta+\theta}\big) + \big(a^{(4,\eta)}_i\nabla_{y_i} \delta v^\eta - \big[a^{(4)}_i\nabla_{y_i}v_t\big]^{\eta+\theta}\big),
		\end{aligned}
	\end{equation}
	were $f_1: \Omega\times\{1,\cdots,N\}\longrightarrow\mathbb{R}$, $f_2: \{v,~\!v_h\}\times\{1,\cdots,N\}\longrightarrow\mathbb{R}$ and $f_3: \{1,\cdots,N\}\longrightarrow\mathbb{R}$, such that $f_1^\eta(y) = a_i^{(1,\eta)}(y),~f_2^\eta(\vartheta) = b_1^\eta\vert \vartheta^\eta\vert_0^2$ and $f_3^\eta = b_2^\eta$, and $D_k$ representing the operators $\nabla_{y_i},~\nabla$ and $\Delta$, for the $k=1,2,3$, respectively.

	In the following theorem, we will present an error estimate for the discrete time associated with the problem \eqref{prob:fixaVarDiscretoExato}:	
		
	\begin{myth}(Error estimates for the fully-discrete problem)
		\label{teo:tempoDiscreto}
		
		Let $v$ be the solution to the problem \eqref{prob:fixaVarDiscretoExato}. Assuming $v_0,v_1\in H_0^2(\Omega)\cap H^4(\Omega)$, and considering the hypotheses of the theorem \ref{teo:tempoContinuo}, we have the following estimate for the approximation error $\forall~\!\theta\in~\! ]1/4,1]$:
		\begin{align*}
			\Vert\delta e\Vert_{L^\infty(0,T;L^2(\Omega))} + \Vert e\Vert_{L^\infty(0,T;H_0^2(\Omega))} \leq C~\!\big[h^2 + (\Delta t)^2\big],
		\end{align*}
		where $C>0$ is constant and independent of $h$ and $\Delta t$.
	\end{myth}	
	
	\begin{proof}				
		The decomposition of the error \eqref{estCont:decompErro}, for the discrete case, is given by $e^\eta = \rho^\eta + \xi^\eta = (v^\eta - R_h v^\eta) + (R_h v^\eta - v_h^\eta)$. Making the difference between the problem \eqref{prob:fixaVarDiscretoExato}  restricted to $V_h$ and its approximate problem, taking in particular $w_h=\delta\xi^\eta$ and using the decomposition of the error and the  Ritz projection, we obtain		
	\begin{equation}\label{estDisc:somaSubProj}
		\Gamma_1^\eta(\xi,v_h) + \Gamma_2^\eta(\xi) = (\psi^\eta,\delta\xi^\eta) - \big[\Gamma_1^\eta(\rho,v) + \Gamma_2^\eta(\rho) + \mu^\eta\big],
	\end{equation}
	where $\Gamma_1: V\times\{v,v_h\}\times\{0,\cdots,N\}\longrightarrow\mathbb{R}$, $~\Gamma_2: V\times\{0,\cdots,N\}\longrightarrow\mathbb{R}~$ and $~\mu:\{0,\cdots,N\}\longrightarrow\mathbb{R}~$ such that
	\begin{align*}	
		\Gamma_1^\eta(\chi,\vartheta) &= \big(\delta^2\chi^\eta,\delta\xi^\eta\big) + a^{(\vartheta,~\!\eta+\theta)}(\chi,\delta\xi^\eta) + \nu~\!\big(\delta\chi^\eta,\delta\xi^\eta\big) +  \big(a^{(4,\eta)}_i(t)\nabla_{y_i}\delta\chi^\eta,\delta\xi^\eta\big),\\
		\Gamma_2^\eta(\chi) 			&= \big(\big[a^{(2)}_{ij}\nabla_{y_i,y_j}\chi\big]^{\eta+\theta},\delta\xi^\eta\big) + \big(\big[a^{(5)}_i\nabla_{y_i}\chi\big]^{\eta+\theta},\delta\xi^\eta\big),\\
		\mu^\eta &= \big(\big[b_1\big(\vert\nabla v_h\vert_0^2 - \vert\nabla v\vert_0^2\big)R_h \nabla v\big]^{\eta+\theta}, \nabla\delta\xi^\eta\big).
	\end{align*}

	\noindent\textbf{Analysis: Terms dependent on $\xi$}
	
	By definition, we have
	\begin{align*}
		\Gamma_1^\eta(\xi,v_h) &= \big(\delta^2\xi^\eta,\delta\xi^\eta\big) + a^{(v_h,~\!\eta+\theta)}(\xi,\delta\xi^\eta) + \nu~\!\big(\delta\xi^\eta,\delta\xi^\eta\big) +  \big(a^{(4,\eta)}_i(t)\nabla_{y_i}\delta\xi^\eta,\delta\xi^\eta\big).
	\end{align*}
	
	Analysing each term, we get
	\begin{align*}
		(\delta^2\xi^\eta,\delta\xi^\eta) = \frac{1}{2\Delta t}\big(|\delta\xi^{\eta+1/2}|^2_0 - |\delta\xi^{\eta-1/2}|^2_0\big).
	\end{align*}
	Summing of $\eta=1$ at $N_0$, for some $1 < N_0 \leq N$, we have
	\begin{equation}\label{estDisc1}
		\sum_{\eta=1}^{N_0-1} (\delta^2\xi^\eta,\delta\xi^\eta) = \frac{1}{2\Delta t}\big(|\delta\xi^{N_0-1/2}|^2_0 - |\delta\xi^{1/2}|^2_0\big).
	\end{equation}
	We also have to
	\begin{equation}\label{estDisc11}
	 \nu~\!\big(\delta \xi^\eta ,\delta\xi^{\eta}\big) = \nu~\!|\delta\xi^\eta|_0^2 \geq 0.	
	\end{equation}
	Integrating by parts and using the definitions \eqref{def:asEbs}, we have the following equality:		
	\begin{equation}\label{estDisc2}
		\big(a^{(4,\eta)}_i\nabla_{y_i}\delta\xi^{\eta},\delta\xi^{\eta}\big) = - \frac{1}{2}\int_\Omega \Big(\nabla_{y_i} a^{(4,\eta)}_i(y)\Big)\big(\delta\xi^{\eta}\big)^{2} dy = \frac{K'~\!^\eta}{K^{\eta}}|\delta\xi^{\eta}|_0^2\geq 0,
	\end{equation}
	where, we use that $\ds\nabla_{y_i} a^{(4,\eta)}_i(y) = -2(K'~\!^\eta/K^\eta)$.

	For to analyse the second term of $\Gamma_1(\xi,v_h)$, we note that
	\begin{equation}\label{estDisc:a*n_f}
		\begin{aligned}
			\big([f~\!D\xi]^{\eta+\theta},D\delta\xi^{\eta}\big) &= \frac{\theta}{2\Delta t}\int_{\Omega}f^{\eta+1}\big(D\xi^{\eta+1}\big)^2 - f^{\eta-1}\big(D\xi^{\eta-1}\big)^2~dy - \theta\big(f'(\sigma_{k1})D\xi^{\eta+1},D\xi^{\eta-1}\big) \\
			&\qquad+ \frac{(1-2\theta)}{2\Delta t}\big[\big(f^\eta D\xi^{\eta+1},D\xi^{\eta}\big) - \big(f^\eta D\xi^{\eta},D\xi^{\eta-1}\big)\big],
		\end{aligned}
	\end{equation}
	for some $\sigma_{k1} \in (t_{\eta-1},t_{\eta+1})$, varying $k\in\{1,\cdots,3\}$. The terms with the derivative in time were obtained from the application of the Theorem Fundamental of Calculus and the Theorem of the Mean Value as follow 
	\begin{equation}\label{explain:TFC_TVM}
		f^{\eta-1} - f^{\eta+1} = -\int_{t_{\eta-1}}^{t_{\eta+1}} f'(t)~dt = -f'(\sigma_{k1})\int_{t_{\eta-1}}^{t_{\eta+1}}\!dt = -2~\!\Delta t ~\!f'(\sigma_{k1}).
	\end{equation}
	
	Therefore by \eqref{estDisc:a*n_f}, we get that
	\begin{align*}
		\begin{aligned}
			&a^{(v_h,~\!\eta+\theta)}(\xi,\delta\xi^\eta) = \sum_{k=1}^3\bigg\{\frac{\theta}{2\Delta t}\int_{\Omega}f_k^{\eta+1}\big(D_k\xi^{\eta+1}\big)^2 - f_k^{\eta-1}\big(D_k\xi^{\eta-1}\big)^2~dy \\
			&\quad- \theta\big(f_k'(\sigma_{k1})D_k\xi^{\eta+1},D_k\xi^{\eta-1}\big) + \frac{(1-2\theta)}{2\Delta t}\big[\big(f_k^\eta D_k\xi^{\eta+1},D_k\xi^{\eta}\big) - \big(f_k^\eta D_k\xi^{\eta},D_k\xi^{\eta-1}\big)\big]\bigg\}.
		\end{aligned}		   	
	\end{align*}
	Summing of $\eta=1$ at $N_0$, we have
	\begin{equation}\label{estDisc34}
		\begin{aligned}			
			&\sum_{\eta=1}^{N_0-1}a^{(v_h,~\!\eta+\theta)}(\xi,\delta\xi^\eta) = \frac{1}{2\Delta t}\sum_{k=1}^3\Big\{\theta\int_{\Omega}f_k^{N_0}\big(D_k\xi^{N_0}\big)^2 + f_k^{N_0-1}\big(D_k\xi^{N_0-1}\big)^2~dy \\
			&\quad + (1-2\theta)\big(f_k^{N_0-1} D_k\xi^{N_0},D_k\xi^{N_0-1}\big)- \theta\int_\Omega f_k^{1}\big(D_k\xi^{1}\big)^2 + f_k^{0}\big(D_k\xi^{0}\big)^2~dy \\
			&\quad- (1-2\theta)\big(f_k^1 D_k\xi^{1},D_k\xi^{0}\big) - \Delta t\sum_{\eta=1}^{N_0-1} \big[2~\!\theta\big(f_k'(\sigma_{k1})D_k\xi^{\eta+1},D_k\xi^{\eta-1}\big) \\
			&\quad- \Delta t~\!(1-2~\!\theta)\big(f_k'(\sigma_{k2})D_k\xi^\eta,D_k\xi^{\eta-1}\big)\big]\Big\},
		\end{aligned}		
	\end{equation}
	for some $\sigma_{k2}\in (t_{\eta-1},t_\eta)$, varying $k\in\{1,\cdots,3\}$.

	So, by analysis \eqref{estDisc1},~\eqref{estDisc2} and \eqref{estDisc34} we obtain
	\begin{equation}\label{estDisc:Gamma1_xi}
		\begin{aligned}	
			&\sum_{\eta=1}^{N_0-1}\Gamma_1^{\eta}(\xi,v_h) \geq \frac{1}{2\Delta t}\Big[\Theta_1^{N_0,N_0-1}-~\!\Theta_1^{1,1} - \Delta t\sum_{\eta=1}^{N_0-1}\Theta_2^\eta\Big],
		\end{aligned}		   	
	\end{equation}
	where
	\begin{align*}
		\Theta_1^{\eta,\lambda} &= \vert\delta\xi^{\eta-1/2}\vert^2_0 + \sum_{k=1}^3\Big[\theta\!\int_{\Omega}f_k^{\eta}\big(D_k\xi^{\eta}\big)^2 + f_k^{\eta-1}\big(D_k\xi^{\eta-1}\big)^2~dy + (1-2\theta)\big(f_k^{\lambda} D_k\xi^{\eta},D_k\xi^{\eta-1}\big)\Big],\\
		\Theta_2^\eta &= \sum_{k=1}^3\big[2~\!\theta\big(f_k'(\sigma_{k1})D_k\xi^{\eta+1},D_k\xi^{\eta-1}\big) + (1-2~\!\theta)\big(f_k'(\sigma_{k2})D_k\xi^\eta,D_k\xi^{\eta-1}\big)\big].
	\end{align*}
	
	Now, analysing each term of $\Gamma_2^\eta(\xi)$, give by
	\begin{align*}
		\Gamma_2^\eta(\xi) = \big(\big[a^{(2)}_{ij}\nabla_{y_i,y_j}\xi\big]^{\eta+\theta},\delta\xi^\eta\big) + \big(\big[a^{(5)}_i\nabla_{y_i}\xi\big]^{\eta+\theta},\delta\xi^\eta\big),
	\end{align*}
	we have, using the inequality of Cauchy-Schwarz and Young, the definition of $\vert\nabla\xi^\eta\vert_0$, increasing the norm $\big[\vert\nabla\xi\vert_0^2\big]^{\eta+\theta}$ and consider $\theta\leq 1$, follows that
	\begin{align*}
		\Gamma_2^\eta(\xi) \leq \frac{\kappa_3}{2}\big[\varepsilon\big(\bar{E}^{\eta+1} + \bar{E}^{\eta} + \bar{E}^{\eta-1}\big) + \frac{1}{\varepsilon}\vert\delta\xi^\eta\vert_0^2\big],
	\end{align*}
	for some $\varepsilon>0$ and for
	\begin{align*}	
		\bar{E}~\!^\eta = \vert\nabla_{y_i}\xi^\eta\vert_0^2 + \vert\nabla\xi^\eta\vert_0^2 + \vert\Delta\xi^\eta\vert_0^2~.
	\end{align*}
	Summing of $\eta=1$	 at $N_0-1$, we get
	\begin{equation}\label{estDisc:Gamma2_xi}
		\sum_{\eta=1}^{N_0-1}\Gamma_2^\eta(\xi) \leq \frac{\kappa_3}{2}~\!\varepsilon~\!\bar{E}^{N_0} + \frac{\kappa_3}{2}\sum_{\eta=1}^{N_0-1}\big[\varepsilon~\!\bar{E}^{\eta} + \frac{1}{\varepsilon}\vert\delta\xi^\eta\vert_0^2\big]
	\end{equation}
	
	\noindent\textbf{Analisys: Terms dependent on $\rho$}	
		
	By definition, we have
	\begin{align*}
		&\Gamma_1^\eta(\rho,v) + \Gamma_2^\eta(\rho) = \big(\delta^2\rho^\eta,\delta\xi^\eta\big) + a^{(v,~\!\eta+\theta)}(\rho,\delta\xi^\eta) + \nu~\!\big(\delta\rho^\eta,\delta\xi^\eta\big) +  \big(a^{(4,\eta)}_i(t)\nabla_{y_i}\delta\rho^\eta,\delta\xi^\eta\big) \\
		&\qquad+ \big(\big[a^{(2)}_{ij}\nabla_{y_i,y_j}\rho\big]^{\eta+\theta},\delta\xi^\eta\big) + \big(\big[a^{(5)}_i\nabla_{y_i}\rho\big]^{\eta+\theta},\delta\xi^\eta\big).
	\end{align*}	
	Analysing the terms of $\Gamma_1^\eta(\rho,v)$ and $\Gamma_2^\eta(\rho)$ applying the inequality of Cauchy-Schwarz and Young, and considering $\theta\leq 1$, follows that
	\begin{align*}		
		\Gamma_1^\eta(\rho,v) + \Gamma_2^\eta(\rho) &\leq \frac{\kappa_4}{2}\Big\{\frac{1}{\varepsilon}\big[\vert\delta^2\rho^\eta\vert_0^2 + \vert\delta\rho^\eta\vert_0^2 + \vert\nabla\delta\rho^\eta\vert_0^2 + \big(\vert\nabla\rho\vert_0^2 + \vert\Delta\rho\vert_0^2\big)^{\eta+1}\\
		&\qquad + \big(\vert\nabla\rho\vert_0^2 + \vert\Delta\rho\vert_0^2\big)^\eta + \big(\vert\nabla\rho\vert_0^2 + \vert\Delta\rho\vert_0^2\big)^{\eta-1}\big] + \varepsilon~\!\vert\delta\xi^\eta\vert_0^2\Big\}\\
		&\leq \frac{\kappa_4}{2}\Big[\frac{1}{\varepsilon}\big(\vert\delta^2\rho^\eta\vert_0^2 + \Vert\delta\rho^\eta\Vert_1^2 + \Vert\rho^{\eta+1}\Vert_2^2 + \Vert\rho^\eta\Vert_2^2 + \Vert\rho^{\eta-1}\Vert_2^2\big) + \varepsilon~\!\vert\delta\xi^\eta\vert_0^2\Big].
	\end{align*}			
	This is similar to \eqref{estCont_Upsilon_rho}. Summing of $\eta=1$ at $N_0$ we get
	\begin{equation}\label{estDisc:Gamma_rho}
		\sum_{\eta=1}^{N_0-1} \big[\Gamma_1^\eta(\rho,v) + \Gamma_2^\eta(\rho) \big] \leq \frac{\kappa_4}{2}\Big\{\frac{1}{\varepsilon}\sum_{\eta=0}^{N_0} \big[\vert\delta^2\rho^\eta\vert_0^2 + \Vert\delta\rho^\eta\Vert_1^2 + \Vert\rho^{\eta}\Vert_2^2\big] + \varepsilon\sum_{\eta=1}^{N_0-1}\vert\delta\xi^\eta\vert_0^2\Big\}.
	\end{equation}
			
	We note that by definition of Ritz projection in \eqref{def:projecaoOrtogonal}, for each discrete time $t=t_{\eta-1},~t_\eta,~t_{\eta+1}$, we get $a^{(v,\eta)}(\rho,\delta\xi^\eta) = a^{(v,\eta)}(v - R_h v,\delta\xi^\eta) = 0$ .\\
		
	\noindent\textbf{Analisys: Nonlinear term}
							
	The calculation procedure for the nonlinear term is analogous to the estimates \eqref{estCont_mu}, that is, for each of the discrete times $t=t_{\eta-1},~t_{\eta},~t_{\eta+1}$, we have		
	\begin{align*}
		\mu^\eta \leq \frac{\kappa_5}{2}\Big[\frac{1}{\varepsilon}\big(\vert\nabla\rho^{\eta+1}\vert_0^2 + \vert\nabla\rho^\eta\vert_0^2 + \vert\nabla\rho^{\eta-1}\vert_0^2\big) + \varepsilon\big(\vert\nabla\xi^{\eta+1}\vert_0^2 + \vert\nabla\xi^{\eta}\vert_0^2 + \vert\nabla\xi^{\eta-1}\vert_0^2 + \vert\delta\xi^{\eta}\vert_0^2\big)\Big].
	\end{align*}
	Summing of $\eta=1$ at $N_0-1$, we get
	\begin{equation}\label{estDisc:nonlinear}
 		\sum_{\eta=1}^{N_0-1}\mu^\eta \leq \frac{\kappa_5}{2}\Big\{\varepsilon~\!\vert\nabla\xi^{N_0}\vert_0^2 + \frac{1}{\varepsilon}\sum_{\eta=0}^{N_0}\big[\Vert\rho^\eta\Vert_1^2\big] + \varepsilon\sum_{\eta=1}^{N_0-1}\big[\vert\nabla\xi^\eta\vert_0^2 + \vert\delta\xi^\eta\vert_0^2\big]\Big\}.
	\end{equation}
	
	\noindent\textbf{Analysis: Term of finite differences}		
	
	By definition, from \eqref{def:formaBilinearDisc}$_2$, we have
	\begin{equation}\nonumber
		\psi^{\eta} = \delta^2 v^\eta - v_{tt}^{\eta+\theta} + \nu~\!\big(\delta v^\eta - v_t^{\eta+\theta}\big) + \Big(a^{(4,\eta)}_i\nabla_{y_i} \delta v^\eta - \big[a^{(4)}_i\nabla_{y_i}v_t\big]^{\eta+\theta}\Big).
	\end{equation}	
	Using integration by parts we can obtain
	\begin{align*}
		\psi^{\eta} &= \frac{1}{(\Delta t)^2}\bigg[\int_{t_{\eta-1}}^{t_{\eta}}v^{(iv)}(\cdot,s)\frac{(s-t_{\eta-1})^3}{3!} ~ds + \int_{t_{\eta}}^{t_{\eta+1}}v^{(iv)}(\cdot,s)\frac{(t_{\eta+1}-s)^3}{3!} ~ds\bigg] \\
		&\quad+ \frac{1}{2\Delta t}\bigg[\int_{t_{\eta-1}}^{t_{\eta}}v'''(\cdot,s)\frac{(s-t_{\eta-1})^2}{2} ~ds + \int_{t_{\eta}}^{t_{\eta+1}}v'''(\cdot,s)\frac{(t_{\eta+1}-s)^2}{2} ~ds\bigg] \\
		&\quad+ \frac{a_i^{(4,\eta)}(y)}{2\Delta t}\bigg[\int_{t_{\eta-1}}^{t_{\eta}}\nabla_{y_i} v'''(\cdot,s)\frac{(s-t_{\eta-1})^2}{2} ~ds + \int_{t_{\eta}}^{t_{\eta+1}}\nabla_{y_i} v'''(\cdot,s)\frac{(t_{\eta+1}-s)^2}{2} ~ds\bigg] \\
		&\quad+ \theta\bigg[\int_{t_{\eta-1}}^{t_{\eta}}\big(v^{(iv)}-v''' + (a_i^{(4)}\nabla_{y_i} v')''\big)(\cdot,s)(s-t_{\eta-1}) ~ds \\
		&\quad\qquad+ \int_{t_{\eta}}^{t_{\eta+1}}\big(v^{(iv)}-v'''+ (a_i^{(4)}\nabla_{y_i} v')''\big)(\cdot,s)(t_{\eta+1}-s) ~ds\bigg].
	\end{align*}			
	
	Considering that $0 \leq (s-t_{\eta-1}) \leq \Delta t$, for all $s\in[t_{\eta-1},t_{\eta}]$ and that $0 \leq (t_{\eta+1}-s) \leq \Delta t$, for all $s\in[t_{\eta},t_{\eta+1}]$, applying the supreme essential in the derivatives of $v$ and then resolving the integrals, considering $\theta\leq 1$, we get
	\begin{equation}\label{estDisc:psi_1}
		\psi^{\eta} \leq \frac{\kappa_{6}}{2}~\!(\Delta t)^2\Big\{\sum_{k=1}^{3}\Big\Vert\frac{d^{k} \nabla v}{dt^{k}}\Big\Vert_{L^\infty(t_{\eta-1},t_{\eta+1})} + \sum_{k=3}^4 \Big\Vert\frac{d^{k} v}{dt^{k}}\Big\Vert_{L^\infty(t_{\eta-1},t_{\eta+1})}\Big\},
	\end{equation}					
	where $\kappa_{6} = 2~\!\kappa_5 + \ds\max_{i\in\{1,\cdots,n\}}\Big\{\max_{t\in[0,T]}\Big[2~\!\sup_{y\in\Omega} \big(|a_i''~\!^{(1)}(y,t)| + |a_i'~\!^{(1)}(y,t)|\big)\Big]\Big\}$. Once that
	\begin{align*}
		(a_i^{(4)}\nabla_{y_i}v')'' &= a_i''~\!^{(4)}\nabla_{y_i}v' + 2~\!a_i'~\!^{(4)}\nabla_{y_i}v'' + a_i^{(4)}\nabla_{y_i}v''' \leq \frac{\kappa_6}{2}\sum_{k=1}^{3}\Big\Vert\frac{d^{k}\nabla v}{dt^{k}}\Big\Vert_{L^\infty(t_{\eta-1},t_{\eta+1})}.
	\end{align*}
	
	Thus, by Cauchy-Schwarz, Young and \eqref{estDisc:psi_1}, we have that
	\begin{align*}
		\big(\psi^\eta,\delta\xi^\eta\big) &\leq \frac{1}{2}\Big[\frac{1}{\varepsilon}\vert\psi^\eta\vert_0^2 + \varepsilon~\!\vert\delta\xi^\eta\vert_0^2\Big] \\
		&\leq \frac{(\kappa_6)^2\!\!}{4~\!\varepsilon}~(\Delta t)^4\Big[\sum_{k=1}^{3}\Big\Vert\frac{d^{k} v}{dt^{k}}\Big\Vert^2_{L^\infty(t_{\eta-1},t_{\eta+1};H^1(\Omega))} + \Vert v^{(iv)}\Vert^2_{L^\infty(t_{\eta-1},t_{\eta+1};L^2(\Omega))}\Big] + \frac{\varepsilon}{2}\vert\delta\xi^\eta\vert_0^2.
	\end{align*}
	Summing of $\eta=1$ at $N_0-1$, we get
	\begin{equation}\label{estDisc:psi}
		\sum_{\eta=1}^{N_0-1} \big(\psi^\eta,\delta\xi^\eta\big) \leq \frac{(\kappa_6)^2\!\!}{4~\!\varepsilon}~\!(\Delta t)^4\Big[\sum_{k=1}^{3}\Big\Vert\frac{d^{k} v}{dt^{k}}\Big\Vert^2_{L^\infty(0,T;H^1(\Omega))} + \Vert v^{(iv)}\Vert^2_{L^\infty(0,T;L^2(\Omega))}\Big] + \frac{\varepsilon}{2}\sum_{\eta=1}^{N_0-1}\vert\delta\xi^\eta\vert_0^2
	\end{equation}
	
	\noindent\textbf{Final estimate}

	Now, passing the term $\Gamma_2^\eta(\xi)$ to the right side of \eqref{estDisc:somaSubProj}, summing of $\eta=1$ at $N_0-1$ and using the estimates \eqref{estDisc:Gamma1_xi} -- \eqref{estDisc:psi}, we obtain	
	\begin{equation}\label{estDisc:final1}
		\frac{1}{2\Delta t}\Big[\Theta_1^{N_0,N_0-1}-\Theta_1^{1,1} - \Delta t\sum_{\eta=1}^{N_0-1}\Theta_2^\eta\Big] \leq \frac{\kappa_7}{2}~\!\Theta_3^\eta
	\end{equation}
	where $\kappa_7 = \kappa_6(1 + \kappa_6/2)~$ and
	\begin{equation}\label{def:Theta3}
		\begin{aligned}		
			&\Theta_3^\eta = \varepsilon~\!\bar{E}^{N_0} + \frac{1}{\varepsilon}\sum_{\eta=0}^{N_0}\big[\vert\delta^2\rho^\eta\vert_0^2 + \Vert\delta\rho^\eta\Vert_1^2 + \Vert\rho^{\eta}\Vert_2^2\big] + \Big(\varepsilon + \frac{1}{\varepsilon}\Big)\sum_{\eta=1}^{N_0-1}\big[\bar{E}^\eta + \vert\delta\xi^\eta\vert_0^2\big]\\
			&\qquad+ \frac{(\Delta t)^4}{\varepsilon}\Big[\sum_{k=1}^{3}\Big\Vert\frac{d^{k} v}{dt^{k}}\Big\Vert^2_{L^\infty(0,T;H^1(\Omega))} + \Vert v^{(iv)}\Vert^2_{L^\infty(0,T;L^2(\Omega))}\Big].
		\end{aligned}
	\end{equation}
	
	We need to analyze the term $\Theta_2^\eta$ before proceeding. We note that, by Cauchy-Schwarz and Young, considering $\theta\leq 1$:
	\begin{align*}
		\Theta_2^\eta &= \sum_{k=1}^3\big[2~\!\theta\big(f_k'(\sigma_{k1})D_k\xi^{\eta+1},D_k\xi^{\eta-1}\big) + (1-2~\!\theta)\big(f_k'(\sigma_{k2})D_k\xi^\eta,D_k\xi^{\eta-1}\big)\big] \\
					  &\leq  2~\!\kappa_2\sum_{k=1}^3\big[\varepsilon\big(\vert D_k\xi^{\eta+1}\vert_0^2 + \vert D_k\xi^{\eta}\vert_0^2\big) + \frac{1}{\varepsilon}\vert D_k\xi^{\eta-1}\vert_0^2\big] = 2~\!\kappa_2\big[\varepsilon\big(\bar{E}^{\eta+1} + \bar{E}^{\eta}\big) + \frac{1}{\varepsilon}\bar{E}^{\eta-1}\big].
	\end{align*}	
	Therefore,
	\begin{equation}\label{estDisc:Theta2}
		\Delta t \sum_{\eta=1}^{N_0-1}\Theta_2^\eta \leq \kappa_7~\!\varepsilon~\!\Delta t~\!\bar{E}^{N_0} + \kappa_8(\varepsilon)~\!\Delta t\sum_{\eta=0}^{N_0-1}\bar{E}^\eta \leq \kappa_7~\!\varepsilon~\!\bar{E}^{N_0} + \kappa_8(\varepsilon)~\!\Delta t\sum_{\eta=0}^{N_0-1}\bar{E}^\eta,
	\end{equation}
	where $\kappa_8(\varepsilon) = \kappa_7(2 + \varepsilon + 1/\varepsilon)$ and in the last inequality we consider $\Delta t\leq 1$ for the first parcel.
	
	On the other hand, for the term $\Theta_1^{1,1}$, considering $\theta\leq 1$ and applying Cauchy-Schwarz and Young, we get
	\begin{equation}\label{estDisc:initialNorms1}
		\begin{aligned}		
			\Theta_1^{1,1} &= |\delta\xi^{1/2}|^2_0 + \sum_{k=1}^3\Big[\theta\!\int_{\Omega}f_k^{1}\big(D_k\xi^{1}\big)^2 + f_k^{0}\big(D_k\xi^{0}\big)^2~dy + (1-2\theta)\big(f_k^{1} D_k\xi^{1},D_k\xi^{0}\big)\Big] \\
			&\leq \kappa_1~\!\Big[\vert\delta\xi^{1/2}\vert_0^2 + \sum_{k=1}^{3}\big[\vert D_k\xi^1\vert_0^2 + \vert D_k\xi^0\vert_0^2\big]\Big] \leq 2~\!\kappa_1~\!\big[\vert\delta\xi^{1/2}\vert_0^2 + \Vert \xi^1\Vert_2^2 + \Vert \xi^0\Vert_0^2\big].
		\end{aligned}
	\end{equation}
	In the last inequality we use that
	\begin{align*}
		\sum_{k=1}^{3}\big[\vert D_k\xi^\eta\vert_0^2\big] = \vert\nabla_{y_i}\xi^\eta\vert_0^2 + \vert \nabla\xi^\eta\vert_0^2 + \vert \Delta\xi^\eta\vert_0^2 \leq 2\big[\vert\xi^\eta\vert_0^2 + \vert \nabla\xi^\eta\vert_0^2 + \vert \Delta\xi^\eta\vert_0^2\big] = 2~\!\Vert\xi^\eta\Vert_2^2, ~\forall~\!\eta.
	\end{align*}

	Following the procedures of article \cite{chou3} and considering \eqref{estDisc:initialNorms1} and the hypothesis \eqref{hip:condInicialContinuo_H3} we get
	\begin{equation}\label{estDisc:initialNorms}
		\Theta_1^{1,1} \leq 2~\!\kappa_1\big[\vert\delta\xi^{1/2}\vert_0^2 + \Vert\xi^1\Vert_2^2 + \Vert\xi^0\Vert_2^2\big] \leq 2~\!\kappa_1\big[(\Delta t)^4 + h^4\big].
	\end{equation}		
	
	Now, analysing the $\rho$ norms of term $\Theta_3^\eta$, we note that, by the Lemma of C\'{e}a and Douglas-Dupont,
	\begin{align*}
		\Vert\rho^\eta\Vert_m \leq  \Vert v^\eta - \hat{v}^\eta\Vert_m \leq C~\!h^{4-m}\Vert v\Vert_{4},
	\end{align*}		
	were $\hat{v}$ is the interpolator of $v$ in $V_h$. So we have, considering $h^4 \leq h^3 \leq h^2$ and multiply by $\Delta t$:
	\begin{equation}\label{estDisc:normaRho}
		\begin{aligned}
			\Delta t\sum_{\eta=0}^{N^*} \Big[\vert\delta^2\rho^\eta\vert_0^2 + \Vert\delta\rho^\eta\Vert_1^2 + \Vert\rho^{\eta}\Vert_2^2\Big] &\leq \int_{0}^T\vert\rho''(s)\vert_0^2 + \Vert\rho'(s)\Vert_1^2 + \Vert\rho(s)\Vert_2^2~ds \\
			&\leq h^4\sum_{k=0}^2 \Big\Vert \frac{d^k~\!v}{d t^k}\Big\Vert_{L^2(0,T;H^4(\Omega))}^2.
		\end{aligned}	
	\end{equation}

	The hypothesis \eqref{hip:regularidades_v's_H2} ensures that the norms of $v$ and its derivatives in the time of \eqref{estDisc:normaRho}, and the norms of last parcels of $\Theta_3^\eta$ in \eqref{def:Theta3} are constants.	
	
	Multiply both side of \eqref{estDisc:final1} by $2~\!\Delta t$ and passing the terms $\Theta_1^{1,1}$ and the sum of $\Theta_2^\eta$ to the right side, and applying the estimates \eqref{estDisc:Theta2} -- \eqref{estDisc:normaRho} we have
	\begin{equation}\label{estDisc:final2}
		\Theta_1^{N_0,N_0-1} \leq \kappa_7~\!\varepsilon~\!\bar{E}^{N_0} + \kappa_9(\varepsilon)\Big\{\big[h^4 + (\Delta t)^4\big] + \sum_{\eta=1}^{N_0-1} \bar{E}^\eta + \vert\delta\xi^\eta\vert_0^2\Big\}.
	\end{equation}
	were $\kappa_9(\varepsilon)$ is the product between $\kappa_8(\varepsilon)$ and maximum of norms of $v$ and its derivatives in time.
	
	By definition, we have
	\begin{align*}
		\vert\delta\xi^\eta\vert_0 = \frac{1}{2}\vert\delta\xi^{\eta+1/2} + \delta\xi^{\eta-1/2}\vert_0 \leq \frac{1}{2}\big[\vert\delta\xi^{\eta+1/2}\vert_0 + \vert\delta\xi^{\eta-1/2}\vert_0\big].
	\end{align*}
	So, summing of $\eta=1$ at $N_0-1$, we get
	\begin{equation}\label{estDisc:deltaxi}
		\sum_{\eta=1}^{N_0-1}\vert\delta\xi^\eta\vert_0 \leq \frac{1}{2}\vert\delta\xi^{N_0-1/2}\vert_0^2 + \sum_{\eta=1}^{N_0-1}\vert\delta\xi^{\eta-1/2}\vert_0.
	\end{equation}
	Therefore substituing \eqref{estDisc:deltaxi} in \eqref{estDisc:final2}, we have
	\begin{equation}\label{estDisc:final3}
		\Theta_1^{N_0,N_0-1} \leq \kappa_7~\!\varepsilon~\!E^{N_0} + \kappa_9(\varepsilon)\Big\{\big[h^4 + (\Delta t)^4\big] + \sum_{\eta=1}^{N_0-1} E^\eta\Big\}.
	\end{equation}
	were
	\begin{align*}
		E^\eta = \vert\delta\xi^{\eta-1/2}\vert_0 + \vert\nabla_{y_i}\xi^\eta\vert_0^2 + \vert\nabla\xi^\eta\vert_0^2 + \vert\Delta\xi^\eta\vert_0^2.
	\end{align*}
	
	Finally, we will analyze the term $\Theta_1^{N_0,N_0-1}$. By definition \eqref{estDisc:Gamma1_xi} we have
	\begin{equation}\nonumber
		\begin{aligned}		
			&\Theta_1^{N_0,N_0-1} = \vert\delta\xi^{N_0-1/2}\vert^2_0 + \sum_{k=1}^3\Big[\theta\int_\Omega f_k^{N_0}\big(D_k\xi^{N_0}\big)^2 + f_k^{N_0-1}\big(D_k\xi^{N_0-1}\big)^2~dy\\
			&\quad+ (1-2\theta)\big(f_k^{N_0-1} D_k\xi^{N_0},D_k\xi^{N_0-1}\big)\Big].
		\end{aligned}	
	\end{equation}
	Adding and subtracting the terms $\theta~\!f_k^{N_0}\big(D_k\xi^{N_0-1}\big)^2$ and $(1-2\theta)\big(f_k^{N_0} D_k\xi^{N_0},D_k\xi^{N_0-1}\big)$, we get
	\begin{equation}\label{estDisc:def_Theta1}
		\begin{aligned}		
			\Theta_1^{N_0,N_0-1} &= \vert\delta\xi^{N_0-1/2}\vert^2_0 + \sum_{k=1}^3\Big[\theta\int_\Omega f_k^{N_0}\big[\big(D_k\xi^{N_0}\big)^2 + \big(D_k\xi^{N_0-1}\big)^2\big]~dy\\
			&\quad+ (1-2\theta)\big(f_k^{N_0} D_k\xi^{N_0},D_k\xi^{N_0-1}\big) - \Theta_4^{N_0,k}\Big]\\
			&= \vert\delta\xi^{N_0-1/2}\vert^2_0 + \sum_{k=1}^3\Big[\theta\int_\Omega f_k^{N_0}\big(D_k\xi^{N_0} - D_k\xi^{N_0-1}\big)^2~dy\\
			&\quad+ \big(f_k^{N_0} D_k\xi^{N_0},D_k\xi^{N_0-1}\big) - \Theta_4^{N_0,k}\Big],
		\end{aligned}	
	\end{equation}
	were, by application of procedures in \eqref{explain:TFC_TVM}, we obtain to follow the equality and the estimate, applying Cauchy-Schwarz and Young, considering $\theta,\Delta t \leq 1$ :
	\begin{align*}
		\Theta_4^{N_0,k} &= \Delta t\Big[\theta\int_\Omega f'_k(\sigma_{k3})\big(D_k\xi^{N_0-1}\big)^2~dy + (1-2~\!\theta)\big(f'_k(\sigma_{k3})D_k\xi^{N_0},D_k\xi^{N_0-1}\big)\Big]\\
		&\leq \kappa_2~\!\vert D_k\xi^{N_0}\vert_0^2 + 2~\!\kappa_2~\!\Delta t\vert D_k\xi^{N_0-1}\vert_0^2.
	\end{align*}		
	for some $\sigma_{k3}\in~\!(t_{N_0-1},t_{N_0})$.

%
%
%
		
	Now, by the polar identity, we have to
	\begin{equation}\label{estDisc:id_polar_caso2}
		\big(f_k^{N_0}D_k\xi^{N_0},D_k\xi^{N_0-1}\big) = \frac{1}{4}\int_\Omega f_k^{N_0}\big[(D_k\xi^{N_0} + D_k\xi^{N_0-1})^2 - (D_k\xi^{N_0} - D_k\xi^{N_0-1})^2\big]dy.
	\end{equation}
	So, replacing \eqref{estDisc:id_polar_caso2} in \eqref{estDisc:def_Theta1} we have
	\begin{equation}\nonumber
		\begin{aligned}
			\Theta_1^{N_0,N_0-1} &\geq \vert\delta\xi^{N_0-1/2}\vert^2_0 + \sum_{k=1}^3\Big[\int_\Omega f_k^{N_0}\Big[\Big(\theta - \frac{1}{4}\Big)\big(D_k\xi^{N_0} - D_k\xi^{N_0-1}\big)^2 \\
			&\quad+ \frac{1}{4}\big(D_k\xi^{N_0} + D_k\xi^{N_0-1}\big)^2 dy - \Theta_4^{N_0,k}\Big].
		\end{aligned}	
	\end{equation}	
	
	Taking $\theta\in~\!]1/4,1]$, we get
	\begin{align}
		\nonumber
		\Theta_1^{N_0,N_0-1} &\geq \vert\delta\xi^{N_0-1/2}\vert^2_0 + \sum_{k=1}^3\big[~\!\kappa_{10}\big(\big\vert D_k\xi^{N_0} - D_k\xi^{N_0-1}\big\vert_0^2 + \big\vert D_k\xi^{N_0} + D_k\xi^{N_0-1}\big\vert_0^2\big) - \Theta_4^{N_0,k}~\!\big]\Big\}\\\label{estDisc:Theta1_caso2}
		&> \kappa_{10}~\!E^{N_0} - \sum_{k=1}^3\big[\Theta_4^{N_0,k}~\!\big],
	\end{align}	
	were $\kappa_{10} = \min\{1,~\kappa_0(\theta-1/4),~\kappa_0/4\} > 0$ and once that 
	\begin{align*}
		\vert D_k\xi^{N_0} - D_k\xi^{N_0-1}\big\vert_0^2 + \big\vert D_k\xi^{N_0} + D_k\xi^{N_0-1}\big\vert_0^2 = 2\big(\vert D_k\xi^{N_0}\vert_0^2 + \vert D_k\xi^{N_0-1}\vert_0^2\big) > \vert D_k\xi^{N_0}\vert_0^2.
	\end{align*}
	
	Substituing \eqref{estDisc:Theta1_caso2} in \eqref{estDisc:final3} and passing the term $\Theta_4^{N_0,k}$ to the right side, we get
	\begin{equation}\label{estDisc:final4}
		(\kappa_{10}-\kappa_7~\!\varepsilon)~\!E^{N_0} \leq \kappa_9(\varepsilon)\Big\{\big[h^4 + (\Delta t)^4\big] + \sum_{\eta=1}^{N_0-1} E^\eta\Big\},~\forall~\!N_0\in\{1,\cdots,N\}.
	\end{equation}
			
	Taking $\varepsilon= (\kappa_{10}/2~\!\kappa_{7}) > 0$, applying the  discrete  Gronwall inequality  and  the equivalence of norms, we have
	\begin{equation}\label{estDisc:Lees}
		\vert\delta\xi^{\eta-1/2}\vert_0^2 + \Vert\xi^{\eta}\Vert_2^2 \leq E^{\eta} \leq C\big[(\Delta t)^4 + h^4\big],~\forall~\eta=1,\cdots,N.
	\end{equation}
	where $\displaystyle C = ({\kappa_{9}}/{\kappa_{10}})\exp\big(({\kappa_{9}}/{\kappa_{10}})T\big)$.

	From \eqref{estDisc:Lees}, using the decomposition of error, the Douglas-Dupont Lemma and taking the  essential supreme in time, we conclude the theorem, i.e.
	\begin{align*}
		\Vert\delta e\Vert_{L^\infty(0,T;L^2(\Omega))} + \Vert e\Vert_{L^\infty(0,T;H_0^2(\Omega))} \leq C~\!\big[h^2 + (\Delta t)^2\big],
	\end{align*}			
	Then, we obtain that the order of convergence of the numerical method is quadratic in time and space.	
\end{proof}

\section{Numerical Simulations}

	In this section, the accuracy of the discrete methods \eqref{newton_n>0} is  tested  by comparing the  approximate numerical solutions with the exact solutions and computing the corresponding approximation error in the norm $L^{\infty}(0,T;L^{2}(0,L))$, for one-dimensional and two-dimensional cases. The results of two numerical examples are presented below.


	In order to compare the approximate and exact solutions, we consider the inhomogeneous versions of equations \eqref{sistema_n>0} and \eqref{sistema_n=0} by introducing the  functions $f(y,t)$, see \eqref{def:MatVet},  in their  right-hand side, respectively. For these new problems, we can construct exact solutions by appropriately choosing the right-hand sides functions.

	In our analysis, we use two solutions for one-dimensional and two-dimensional cases, and two examples of moving boundarys, as shown in Table \ref{tab:samples} and Table \ref{tab:boundaries}, respectively.
		\begin{table}[htbp]
		    \centering		    
			\begin{tabular}{lcc}
				\toprule
				\multicolumn{3}{c}{$v(y,t)$}\\								
				\midrule
				 & \multicolumn{1}{c}{$\mathbb{R}$} & \multicolumn{1}{c}{$\mathbb{R}^2$}\\
				\midrule
   			    S$_1$ & $10^{-1}(y^2-1)^2\cos(2\pi t)$ & $10^{-1}\big[(y_1^2-1)(y_2^2-1)\big]^2\cos(2\pi t)$ \\
				S$_2$ & $10^{-3}(y^2-1)^2\sin(2\pi t)$ & $10^{-7}\big[(y_1^2-1)(y_2^2-1)\big]^2\sin(2\pi t)$ \\
				\bottomrule
			\end{tabular}
			\caption{Examples of exact solutions $v(y,t)$}
			\label{tab:samples}
		\end{table}			
		\begin{table}[htbp]
		    \centering
			\begin{tabular}{lcc}
				\toprule
				\multicolumn{3}{c}{$K(t)$}\\								
				\midrule
					 & \multicolumn{1}{c}{$\mathbb{R}$} & \multicolumn{1}{c}{$\mathbb{R}^2$}\\
				\midrule
				  B$_1~$ & $64 + t/2^7$       & $64 + t/2^{17}$\\
				  B$_2~$ & $64 + 2(1-e^{-t})$ & $64 + (1-e^{-t})/2^{17}$\\
				\bottomrule
			\end{tabular}
			\caption{Examples of Moving boundaries $K(t)$}
			\label{tab:boundaries}
		\end{table}				
	We emphasize that the boundaries of the Table \ref{tab:boundaries}, satisfy the hypothesis \eqref{hip:dominio2D_H1}. The initial conditions immediately follow the exact solutions given.	
	
	The Theorem \eqref{teo:tempoDiscreto}, shows that the solution is unconditionally convergent with order of quadratic convergence in space and time for $\forall~\!\theta\in~\!]1/4, 1]$. However, we know from the literature that for  $\forall~\!\theta\in [0, 1/4[$, the system is conditionally convergent and we can see this fact in numerical simulations in one-dimensional case for the example S$_1$ and moving boundary B$_1$, see Table \ref{tab:tetaAnalisys}.
In addition, we can see that the error is inversely proportional to the parameter $\forall~\!\theta\in~\!]1/4, 1]$,  that is,  the smaller $\theta$ the minor error, as can be observed in the Table \ref{tab:tetaAnalisys}.
		\begin{table}[htbp]
			\centering
			\begin{tabular}{l ccccc}
				\toprule
				\multicolumn{6}{l}{$E_{L^{\infty}(0,T;L^2(\Omega))}$}\\
				\midrule
  	    $h\diagdown\theta$  & $0$        & $0.25$     & $0.5$      & $0.75$     & $1$        \\
				\midrule
					$2^{-1}$ & $5.495e-3$ & $5.463e-3$ & $5.434e-3$ & $5.406e-3$ & $5.379e-3$ \\
					$2^{-2}$ & $2.759e-3$ & $2.750e-3$ & $2.745e-3$ & $2.743e-3$ & $2.745e-3$ \\
					$2^{-3}$ & $6.404e-4$ & $6.417e-4$ & $6.572e-4$ & $6.859e-4$ & $7.264e-4$ \\
					$2^{-4}$ & $2.016e-4$ & $2.076e-4$ & $2.532e-4$ & $3.220e-4$ & $4.022e-4$ \\
					$2^{-5}$ & $4.266e-5$ & $5.507e-5$ & $1.510e-4$ & $2.472e-4$ & $3.434e-4$ \\
					$2^{-6}$ & diverge    & $5.380e-5$ & $1.498e-4$ & $2.460e-4$ & $3.422e-4$ \\					
				\bottomrule									
			\end{tabular}
			\caption {Error with $\Delta t = 2^{-7}$ fixed and varying $\theta$, for $\Omega\in\mathbb{R}$, example S$_1$ and boundary B$_1$}
			\label{tab:tetaAnalisys}
		\end{table}	

	Let's set, from now on, $\theta = 1/4$, for generating a minor error. In the Table \ref{tab:errorL2_deltat}, the parameter $h = 2^{-6}$ is fixed and the $\Delta t$ value are varying in $\Delta t =2^{-(i+1)}, ~i=1,\cdots,6$.
	\begin{table}[htbp]
		\centering
		\begin{tabular}{c cc c cc}
			\toprule
			\multicolumn{6}{l}{$E_{L^{\infty}(0,T;L^2(\Omega))}$}\\
			\midrule
							  & \multicolumn{2}{c}{$\mathbb{R}$} & & \multicolumn{2}{c}{$\mathbb{R}^2$} \\ \cmidrule(lr){2-3} \cmidrule(lr){5-6}
					 		  & B$_{1}$    & B$_{2}$     
					 		  & 
					 		  & B$_{1}$    & B$_{2}$    \\
			\midrule
						  	  & $9.651e-2$ & $9.647e-2$  
						  	  &
						  	  & $8.711e-3$ & $8.711e-3$  
						  	  \\
							  & $1.750e-2$ & $1.757e-2$  
							  &
							  & $1.577e-3$ & $1.577e-3$  
							  \\
					  S$_1$   & $3.173e-3$ & $3.200e-3$  
						  	  &
						  	  & $2.853e-4$ & $2.853e-4$  
						  	  \\
							  & $7.713e-4$ & $7.670e-4$  
							  &
							  & $6.954e-5$ & $6.954e-5$  
							  \\
							  & $2.063e-4$ & $2.051e-4$  
							  &
							  & $1.857e-5$ & $1.857e-5$  
							  \\
							  & $5.380e-5$ & $5.350e-5$  
							  &
							  & $4.837e-6$ & $4.837e-6$  
							  \\					  
			\midrule	
						  	  & $3.181e-4$ & $3.169e-4$  
						  	  &
						  	  & $9.562e-2$ & $9.570e-2$  
						  	  \\
							  & $7.296e-5$ & $7.273e-5$  
							  &
							  & $2.385e-2$ & $2.386e-2$  
							  \\
					  S$_2$   & $1.700e-5$ & $1.695e-5$  
						  	  &
						  	  & $5.978e-3$ & $5.980e-3$  
						  	  \\
							  & $4.072e-6$ & $4.060e-6$  
							  &
							  & $1.515e-3$ & $1.514e-3$  
							  \\
							  & $9.946e-7$ & $9.917e-7$  
							  &
							  & $3.805e-4$ & $3.802e-4$  
							  \\
							  & $2.457e-7$ & $2.450e-7$  
							  &
							  & $3.805e-4$ & $3.802e-4$  
							  \\
			\bottomrule									
		\end{tabular}
		\caption{Error with $h = 2^{-6}$ fixed  and $\Delta t=2^{-(i+1)}, ~i=1,\cdots,6$}
		\label{tab:errorL2_deltat}
	\end{table}	

	In the Table \ref{tab:errorL2_h} is the opposite, the parameter $\Delta t = 2^{-7}$ is fixed and the $h$ value are varying  $ h=2^{-i}, ~i=1,\cdots,6$.

	\begin{table}[htbp]
		\centering
		\begin{tabular}{c cc c cc}
			\toprule
			\multicolumn{6}{l}{$E_{L^{\infty}(0,T;L^2(\Omega))}$}\\
			\midrule
							  & \multicolumn{2}{c}{$\mathbb{R}$} & & \multicolumn{2}{c}{$\mathbb{R}^2$} \\ \cmidrule(lr){2-3} \cmidrule(lr){5-6}
					 		  & B$_{1}$    & B$_{2}$     
					 		  & 
					 		  & B$_{1}$    & B$_{2}$    \\
			\midrule
						  	  & $5.463e-3$ & $5.462e-3$  
						  	  &
						  	  & $9.895e-3$ & $9.895e-3$  
						  	  \\
							  & $2.750e-3$ & $2.751e-3$  
							  &
							  & $2.105e-3$ & $2.105e-3$  
							  \\
					  S$_1$   & $6.417e-4$ & $6.419e-4$  
						  	  &
						  	  & $6.283e-4$ & $6.283e-4$  
						  	  \\
							  & $2.076e-4$ & $2.077e-4$  
							  &
							  & $1.889e-4$ & $1.889e-4$  
							  \\
							  & $5.507e-5$ & $5.477e-5$  
							  &
							  & $5.837e-5$ & $5.837e-5$  
							  \\
							  & $5.380e-5$ & $5.350e-5$  
							  &
							  & $4.837e-5$ & $4.837e-5$  
							  \\					  
			\midrule	
						  	  & $5.254e-5$ & $5.254e-5$  
						  	  &
						  	  & $1.000e-9$ & $1.000e-9$  
						  	  \\
							  & $2.745e-5$ & $2.745e-5$  
							  &
							  & $2.200e-10$ & $2.200e-10$  
							  \\
					  S$_2$   & $6.383e-6$ & $6.383e-6$  
						  	  &
						  	  & $7.000e-11$ & $7.000e-11$  
						  	  \\
							  & $1.984e-6$ & $1.984e-6$  
							  &
							  & $2.335e-11$ & $2.335e-11$  
							  \\
							  & $2.460e-7$ & $2.453e-7$  
							  &
							  & $2.205e-11$ & $2.205e-11$  
							  \\
							  & $2.457e-7$ & $2.450e-7$  
							  &
							  & $2.001e-11$ & $2.001e-11$  
							  \\
			\bottomrule									
		\end{tabular}
		\caption{Error with $\Delta t = 2^{-7}$ fixed and $h=2^{-i}, ~i=1,\cdots,6$}		
		\label{tab:errorL2_h}
	\end{table}		


	Finally, we present the order of numerical convergence in the Table \ref{tab:errorL2}, $h=2\Delta t$  and varying $\Delta t=2^{-(i+1)}, ~i=1,\cdots,6$, which coincides with the expected theoretical results
obtained in Theorem \ref{teo:tempoContinuo} and Theorem \ref{teo:tempoDiscreto} demonstrated in this paper, what suggests to us, that these results are valid also for $\theta = 1/4$.	

	\begin{table}[htbp]
		\centering
		\begin{tabular}{c cc c cc}
			\toprule
			\multicolumn{6}{l}{Convergence Rate}\\
			\midrule
							  & \multicolumn{2}{c}{$\mathbb{R}$} & & \multicolumn{2}{c}{$\mathbb{R}^2$} \\ \cmidrule(lr){2-3} \cmidrule(lr){5-6}
					 		  & B$_{1}$    & B$_{2}$     
					 		  & 
					 		  & B$_{1}$    & B$_{2}$    \\
			\midrule
						  	  & $-$ & $-$  
						  	  &
						  	  & $-$ & $-$  
						  	  \\
							  & $2.466$ & $2.460$  
							  &
							  & $2.208$ & $2.208$  
							  \\
					  S$_1$   & $2.434$ & $2.429$  
						  	  &
						  	  & $2.475$ & $2.475$  
						  	  \\
							  & $2.014$ & $2.031$  
							  &
							  & $2.095$ & $2.095$  
							  \\
							  & $1.953$ & $1.954$  
							  &
							  & $1.884$ & $1.884$  
							  \\
							  & $1.947$ & $1.947$  
							  &
							  & $1.955$ & $1.955$  
							  \\					  
			\midrule	
						  	  & $-$ & $-$  
						  	  &
						  	  & $-$ & $-$  
						  	  \\
							  & $2.087$ & $2.085$  
							  &
							  & $1.552$ & $1.552$  
							  \\
					  S$_2$   & $2.093$ & $2.093$  
						  	  &
						  	  & $1.922$ & $1.922$  
						  	  \\
							  & $2.061$ & $2.061$  
							  &
							  & $2.002$ & $2.002$  
							  \\
							  & $2.037$ & $2.037$  
							  &
							  & $1.982$ & $1.982$  
							  \\
							  & $2.018$ & $2.018$  
							  &
							  & $2.009$ & $2.009$  
							  \\
			\bottomrule									
		\end{tabular}
		\caption{Convergence rate with $h=2\Delta t$ and $\Delta t= 2^{-(i+1)}, ~i=1,\cdots,6$}	
		\label{tab:errorL2}
	\end{table}		

	\subsection{Asymptotic Behavior}
	
	In Figure \ref{fig:energys_1D} we present the decay of the energy of the homogeneous solution of the mobile problem \eqref{prob:movel} for all $t\in[0,T^*]$, were $T^*$ is the time when the energy is of the order of $10^{-10}$. We consider the initial conditions taking $t=0$ in the exact functions of the Table \ref{tab:samples} and the moving borders of the Table \ref{tab:boundaries}. 
	\begin{figure}[htb]
		\centering
		\subfigure[Example S$_1$]{\includegraphics[scale=.425]{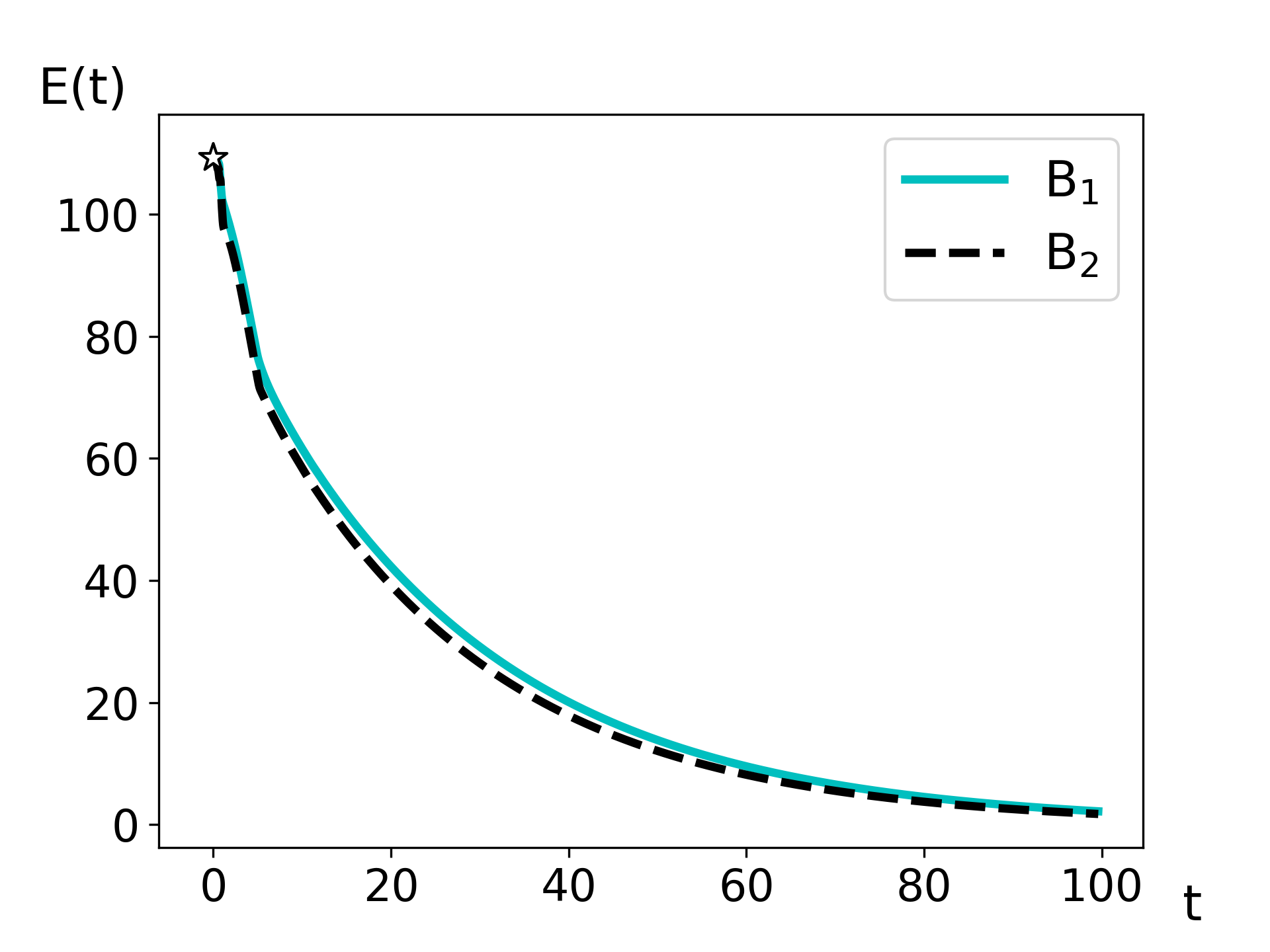}}
		\subfigure[Example S$_2$]{\includegraphics[scale=.425]{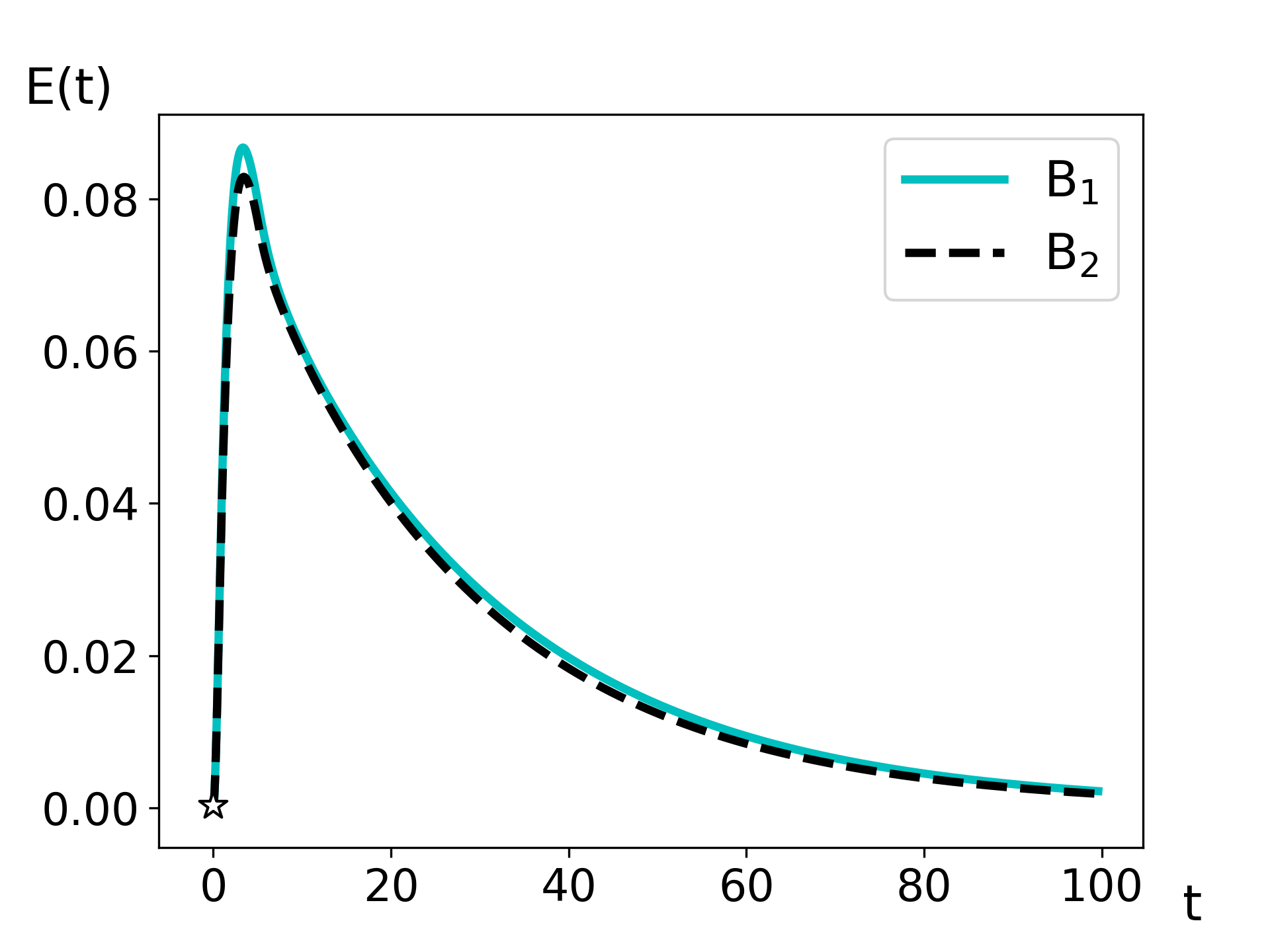}}
		\caption{Asymptotic behavior of energy}	
		\label{fig:energys_1D}		
	\end{figure}
	
	In the simulations we obtained an energy of the order of $10^{-10}$ at the approximate times $T^*=770$, in example S$_1$, and $T^*=560$, in example S$_2$ , for both moving borders. In Figure \ref{fig:fronteiras} we present the evolution of the moving boundaries.
	\begin{figure}[htb]
		\centering	
		\includegraphics[scale=0.425]{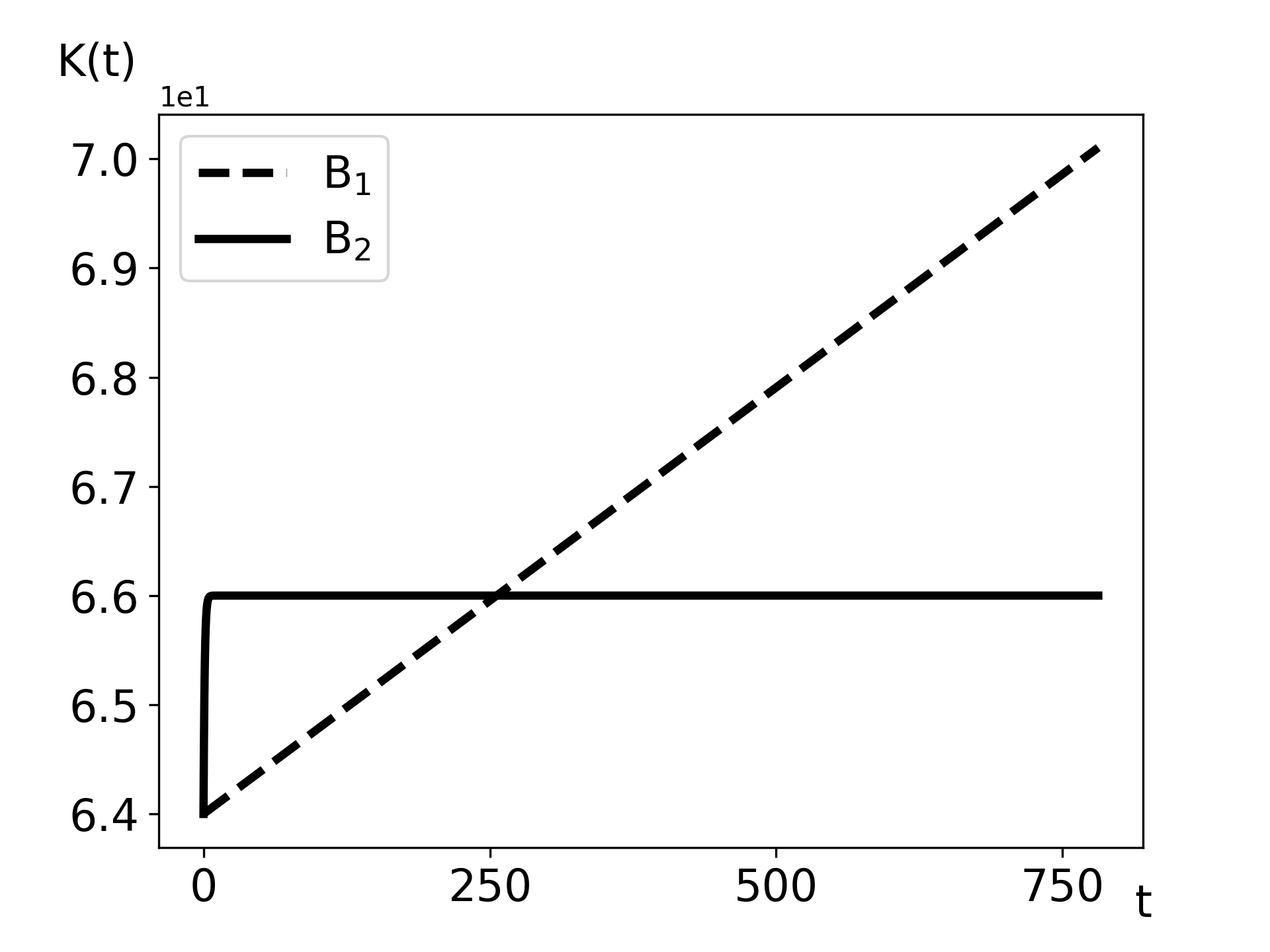}
		\caption{Moving boundaries}
		\label{fig:fronteiras}				
	\end{figure}

	\subsection{Homogeneos Solution}	
	
	Now we will present the evolution of the numerical solution of the problem \eqref{prob:fixa}, obtained from the \eqref{newton_n>0}, by considering $f(y,t) = 0$ and the initial conditions from taking $t=0$ in exact functions of Table \ref{tab:samples}.	
	
	In the simulations were considered $h =2^{-6}$ and $\Delta t=2^{-7}$. The following parameters were prefixed $\zeta_0 = 128$, $\zeta_1 = 2$ and $\nu = 1$.			
	\begin{figure}
		\centering
		\includegraphics[scale=0.425]{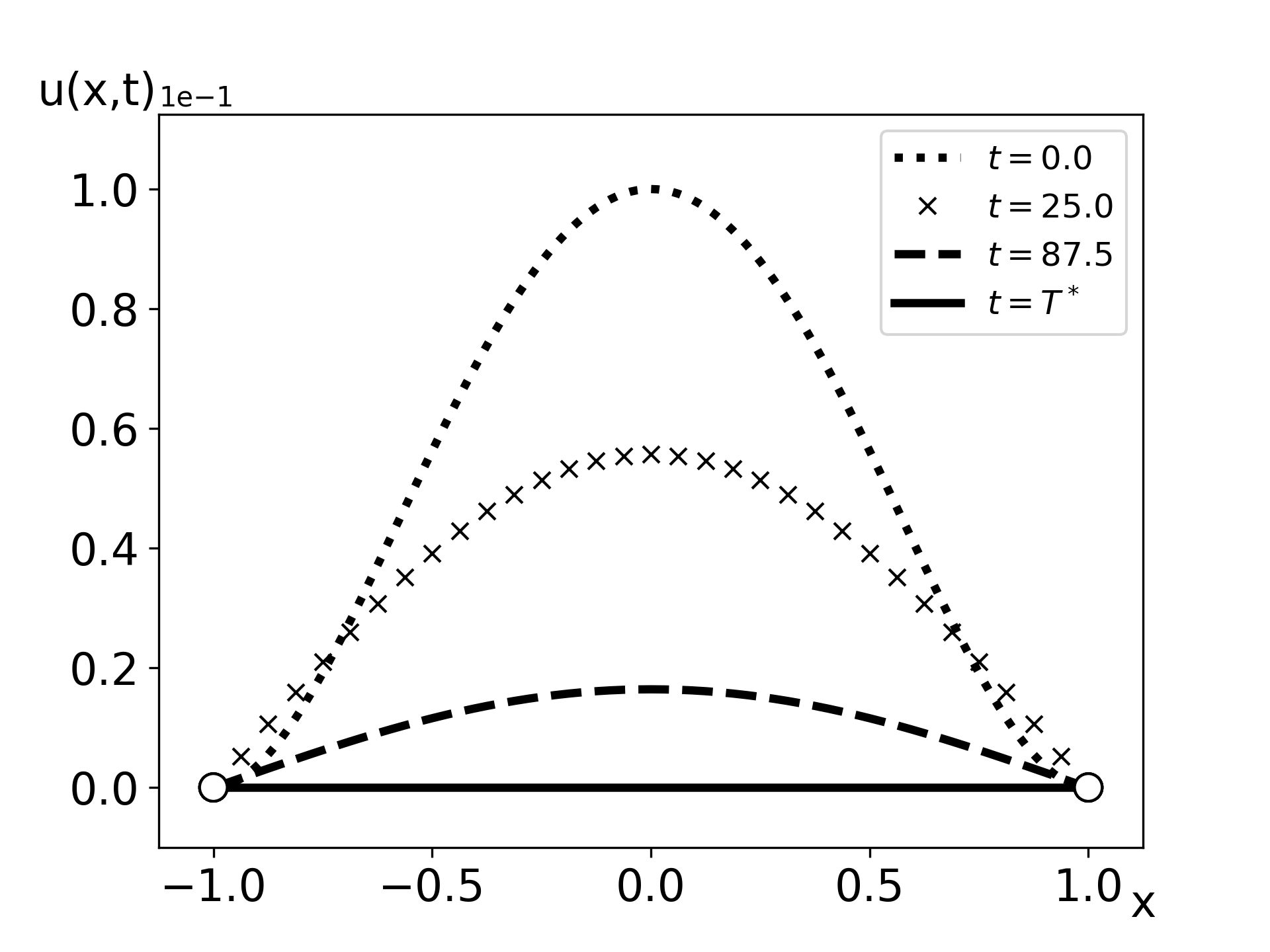}
		\caption{Homogeneous solutions for example S$_1$}
		\label{fig:solution_S1}
	\end{figure}
	\begin{figure}
		\centering
		\subfigure[Under the effect of the initial velocity]{\includegraphics[scale=0.425]{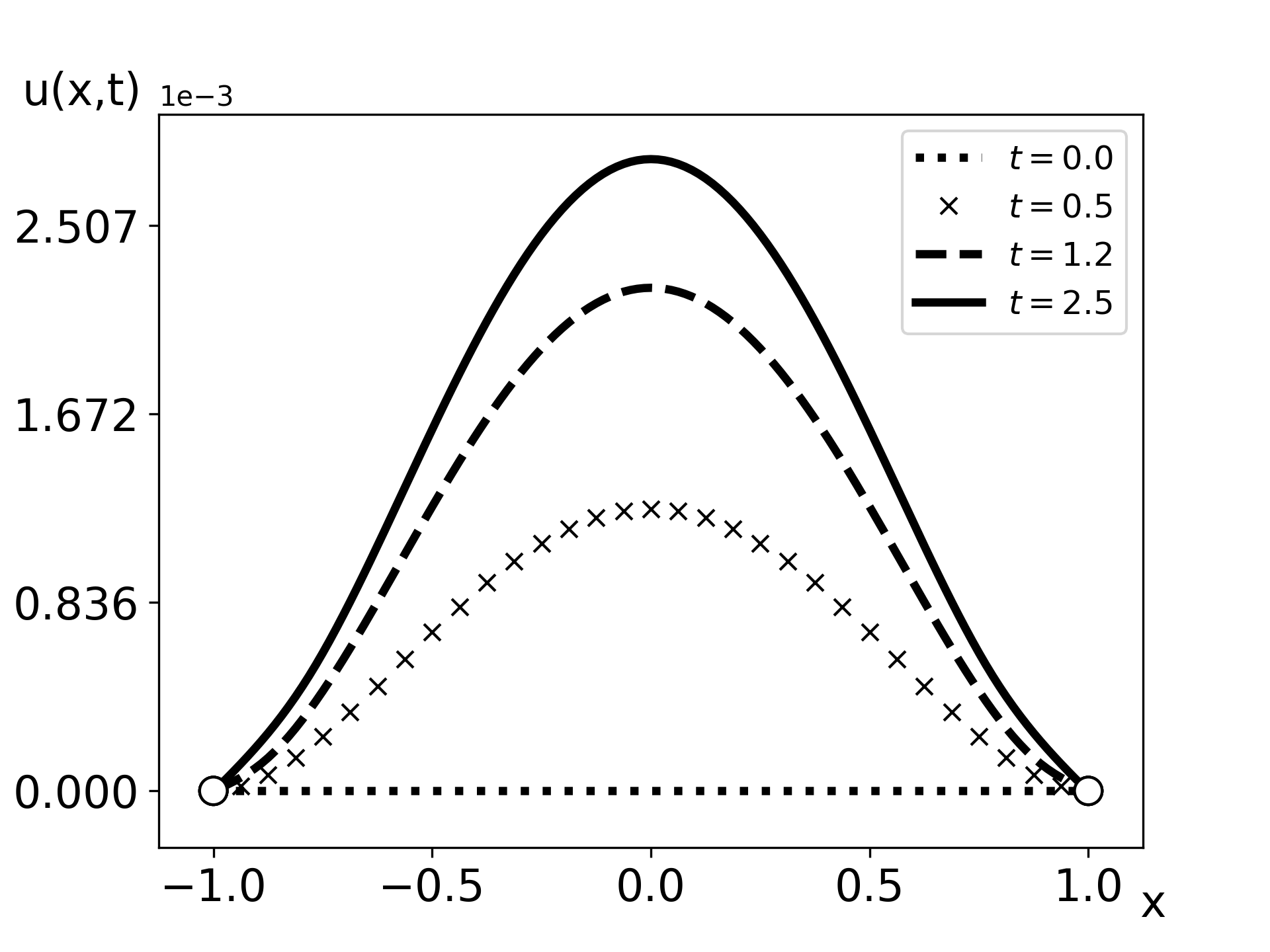}}
		\subfigure[Post-effect of initial velocity]{\includegraphics[scale=0.425]{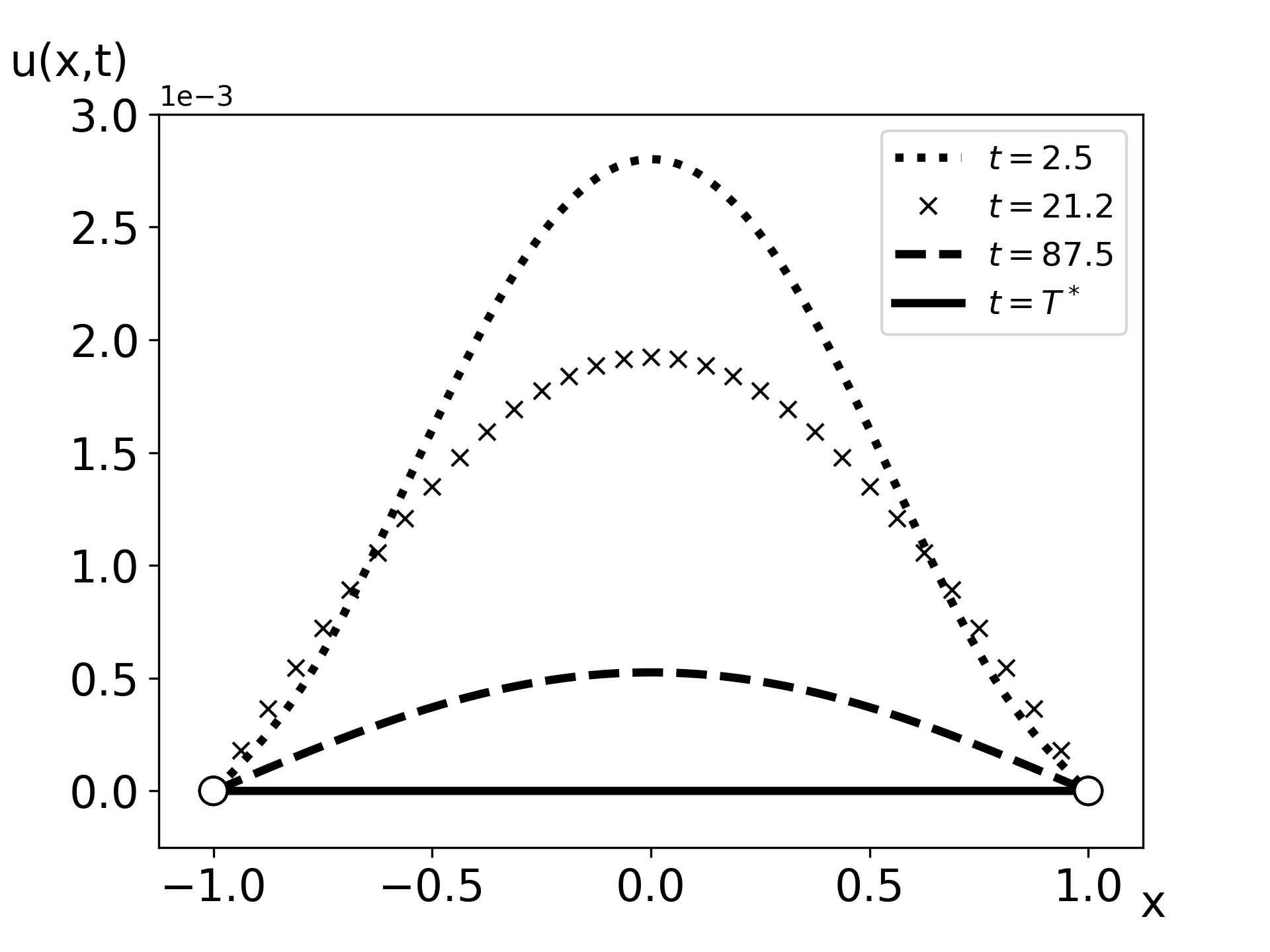}}
		\caption{Homogeneous solutions for example S$_2$}
		\label{fig:solution_S2_1D}		
	\end{figure}	

	
\section{Conclusions}\label{sec:conc}

	In this paper, we focused our study on the numerical analysis of Problem \eqref{edp:movel}, introducing a numerical method and demonstrating the order of convergence for both the semi-discrete problem and the totally discrete problem. Initially, we announced the known theoretical results in the presence of moving boundary and damping. In addition we announce the asymptotic decay.
We introduce a family of numeric methods, based on a theta parameter, and show numerically that for $\theta\in[0, 1/4[$, the numerical method is conditionally convergent (as expected) and for $\theta\in[1/4, 1]$ the convergence is unconditionally convergent (see the Tables when $h$  or $\Delta t$ are varying). In addition, we can verify that the numerical error is smaller the smaller the $\theta$ limited to $[1/4, 1]$. Finally, let us show the decay of the system's energy.

	It is worth emphasizing that the presented numerical error results were in complete agreement with the results presented in the theoretical part, and they also indicate the reliability of this numerical method for obtaining good approximate solutions of these nonlinear problems.

\end{document}